%% file: mainRev.tex
\documentclass[a4paper]{elsarticle}
\usepackage[T1]{fontenc}
\usepackage[utf8x]{inputenc}
\usepackage[bookmarksnumbered=true]{hyperref}
\usepackage{algorithm}
\usepackage{algorithmic}
\usepackage{amsmath}
\usepackage{amssymb}
\usepackage{amsfonts}
\usepackage{amsthm}
\usepackage{cases}
\usepackage{booktabs}
\usepackage{fancyhdr}
\usepackage{graphicx}
\usepackage{multicol}
\usepackage{multirow}
\usepackage{placeins}
\usepackage{color}
\usepackage{subfigure}
\usepackage{nth}
\usepackage{lscape}
\usepackage{pgfplotstable}

\newcommand{\acuta}{ACUTA}
\newcommand{\dm}{DM}

\newcommand{\grip}{GRIP}

\newcommand{\uta}{UTA}
\newcommand{\utap}{UTA-poly}
\newcommand{\utasp}{UTA-splines}
\newcommand{\utadis}{UTADIS}
\newcommand{\utadisp}{UTADIS-poly}
\newcommand{\utadissp}{UTADIS-splines}
\newcommand{\utagms}{UTA-GMS}
\newcommand{\utastar}{UTA-STAR}

\newcommand{\transp}{\mathsf{T}}
\newcommand{\mcda}{MCDA}
\newcommand{\sos}{SOS}

\newtheorem{theorem}{Theorem}
\newtheorem{lemma}[theorem]{Lemma}

\begin{document}

\begin{frontmatter}
\title{\utap{} and \utasp{}: additive value functions with polynomial
marginals}
\author[umons,centrale]{Olivier Sobrie}
\ead{olivier.sobrie@gmail.com}
\author[umons]{Nicolas Gillis}
\ead{nicolas.gillis@umons.ac.be}
\author[centrale]{Vincent Mousseau}
\ead{vincent.mousseau@centralesupelec.fr}
\author[umons]{Marc Pirlot}
\ead{marc.pirlot@umons.ac.be}

\address[umons]{Université de Mons, Faculté Polytechnique, 9 rue de
Houdain, 7000 Mons, Belgium}
\address[centrale]{CentraleSupélec, Laboratoire Génie Industriel,
Grande Voie des Vignes, 92295 Châtenay-Malabry, France}

\begin{abstract}
\input{abstractRev.tex}
\end{abstract}


\end{frontmatter}

\input{introRev.tex}
\input{utaRev.tex}
\input{utapRev.tex}
\input{utaspRev.tex}
\input{exampleRev.tex}
\input{experimentsRev.tex}
\input{concluRev.tex}

\FloatBarrier
\section*{References}
\bibliographystyle{plain}
\bibliography{mcda,pl,sdp}

\FloatBarrier 
\begin{appendix}
\input{appendixRev.tex}
\end{appendix}

\end{document}

%% file: abstractRev.tex
% vim:spell spelllang=en

Additive utility function models are widely used in multiple criteria
decision analysis.
In such models, a numerical value is associated to each alternative
involved in the decision problem.
It is computed by aggregating the scores of the alternative on the
different criteria of the decision problem.
The score of an alternative is determined by a marginal value function
that evolves monotonically as a function of the performance of the
alternative on this criterion.
Determining the shape of the marginals is not easy for a decision maker.
It is easier for him/her to make statements such as ``alternative $a$ is
preferred to $b$''.
In order to help the decision maker, \uta{} disaggregation procedures use linear programming to approximate the marginals by piecewise linear
functions based only on such statements.
In this paper, we propose to infer polynomials and splines instead of piecewise linear
functions for the marginals. In this aim, we use semidefinite programming instead of linear programming.
We illustrate this new elicitation method and present some experimental
results.

%% file: introRev.tex
% vim:spell spelllang=en

\section{Introduction}

The theory of value functions aims at assigning a number to each alternative in such a way that the decision maker's preference order on the alternatives is the same as the order on the numbers associated with the alternatives. The number or value associated to an alternative is a monotone function of its evaluations on the various relevant criteria. For preferences satisfying some additional properties (including \textit{preferential independence}), the value of an alternative can be obtained as the sum of marginal value functions each depending only on a single criterion \citep[][Chapter 6]{keeneyraiffa1976}.

These functions usually are monotone, i.e., marginal value functions either
increase or decrease with the assessment of the alternative on the associated criterion.
Many questioning protocols have been proposed aiming to elicit an additive value function \citep{keeneyraiffa1976,fishburn67} through interactions with the decision maker (DM). These \textit{direct} elicitation methods are time-consuming and require a substantial cognitive effort from the DM. Therefore, in certain cases, an indirect approach may prove fruitful. The latter consists in \textit{learning} an additive value model (or a set of such models) from a set of declared or observed preferences. In case we know that the DM prefers alternative $a^{i}$ to $b^{i}$ for some pairs $(a^{i}, b^{i}),\ i= 1, 2, \ldots$, we may infer a model that is compatible with these preferences. Learning approaches have been proposed not only for inferring an additive value function that is used to rank all other alternatives. They have also been used for sorting alternatives in ordered categories \citep{yu1992,roybouyssou1993,zopounidisdoumpos2002}. In this model, an alternative is assigned to a category (e.g.\ ``Satisfactory'', ``Intermediate'', ``Not satisfactory'') whenever its value passes some threshold and does not exceed some other, which are respectively the lower and upper values of the alternatives to be assigned to this category.

The \uta{} method \citep{jaquetlsiskos1982} was the original proposal for this purpose. % ***NG*** What purpose? Cette phrase n'est pas très claire.
It uses a linear programming formulation to determine piecewise linear marginal value functions that are compatible with the DM's known preferences.
Several variants of this idea for learning a piecewise linear additive value function on the basis of examples of ordered pairs of alternatives  are described in \cite{jaquetlsiskos2001}. The variant used for inferring a rule to assign alternatives to ordered categories on the basis of assignment examples is called \utadis{} in \cite{zopounidisdoumpos99} (see also \cite{zopounidisdoumpos2002}). The interested reader is referred to \cite{SiskosInErgFigGre05} for a comprehensive review of \uta{} methods, their variants and developments.

A problem with these methods is that, often, the information available about the DM's preferences is far from determining a single additive value function. In general, the set of piecewise linear value functions compatible with the partial knowledge of the DM's preferences is a polytope in an appropriate space. Therefore the learning methods that have been proposed either select a ``representative'' value function or they work with all possible value functions and derive \textit{robust} conclusions, i.e.\ information on the DM's preference that does not depend on the particular choice of a value function in the polytope.
Among the latter, one may cite \utagms{} \citep{grecoetal2008,grecoetal2010} and \grip{} \citep{figueiraetal2009b}. This research avenue is known under the name  \textit{robust ordinal regression} methods.

The original approach has to face the issue of defining what is a ``representative'' value function or a default value function. \utastar{} \citep{jaquetlsiskos1982,siskosyanacopoulos1985} solves the problem implicitly by returning an ``average solution'' computed as the mean of ``extreme'' solutions (this approach is sometimes referred to as ``post-optimality analysis'' \citep{DoumposEtAl2014}). Although, \utastar{} does not give any formal definition of a representative solution, it returns a solution that tends to lie ``in the middle'' of the polytope determined by the constraints. The idea of centrality, as a definition of representativeness, has been illustrated with the \acuta{} method \cite{bousetal2010}, in which the selected value function corresponds to the analytic center of the polytope, and the other formulation, using the Chebyshev center \citep{DoumposEtAl2014}. On the other hand, \cite{KadzinskiEtal2012} propose a completely different approach to the idea of representativeness. They define five targets and select a representative value function taking into account a prioritization of the targets by the DM in the context of robust ordinal regression methods.  The same authors also proposed a method for selecting a representative value function for robust sorting of alternatives in ordered categories \citep{GrecoEtAlReprVFSorting2011}.

In all the approaches aiming to return a ``representative'' value function, the marginal value functions are piecewise linear. The choice of such functions is historically motivated by the opportunity of using linear programming solvers (except for \acuta{} \cite{bousetal2010}). Although piecewise linear functions are well-suited for approximating monotone continuous functions, their lack of smoothness (derivability) may make them seem ``not natural'' in some contexts, especially for economists.  Brutal changes in slope at the breakpoints is difficult to explain and justify. Therefore, using smooth functions as marginals is advantageous from an interpretative point of view.

The MIIDAS system \cite{siskosetalejor99} proposes tools to model marginal value functions. Possibly non-linear (and even non-monotone) shapes of marginals can be chosen from parameterized families of curves. The value of the parameters is adjusted by using \textit{ad hoc} techniques such as the midpoint value.  In \cite{burgeraetal2002}, the authors propose an inference method based
on a linear program that infers quadratic utility functions in the context of an application to the banking sector.

In this paper, we propose another approach to build the marginals, which is
based on \textit{semidefinite programming}. It allows for learning marginals which are composed of one or several
polynomials of degree $d$, $d$ being fixed a priori. Besides facilitating the interpretations of the returned marginals, using such functions increases the descriptive power of the model, which is of secondary importance for decision aiding but may be valuable in other applications. In particular, in machine learning, learning sets may involve thousands of pairs of ordered alternatives or assignment examples, which may provide an advantage to more flexible models. Beyond these advantages, the most striking aspect of this work is the fact that a single new optimization technique allows us to deal with polynomial of any degree and piecewise polynomial marginals instead of piecewise linear marginals. The semidefinite programming approach used in this paper for UTA might open new perspectives for the elicitation of other preference models based on additive or partly additive value structures, such as additive differences models (MACBETH \cite{bana1994macbeth,bana2005}), and GAI networks \citep{GonzalesEtAl2011}.

This paper contributes to the field of preference elicitation by proposing a new way to model marginal value functions using polynomials or splines instead of piecewise linear value functions. The paper is organized as follows.
Section 2 recalls the principles of \uta{} methods.
We then describe a new method called \utap{} which computes each marginal as a degree $d$
polynomial instead of a piecewise linear function.
Section 4 introduces another approach called \utasp{} which is a
generalization of \uta{} and \utap{}. The shape of the marginals used by \utasp{} are piecewise polynomials or polynomial splines. These methods can be used either for ranking alternatives or for sorting them in ordered categories.
The next section gives an illustrative example of the use of \utap{} and \utasp{}.
Finally, we present experimental results comparing the new methods with \uta{} both in terms of accuracy, model retrieval and computational effort.

%% file: utaRev.tex
% vim:spell spelllang=en

\section{UTA methods}

In this section we briefly recall the basics of the additive value function model (see \cite{keeneyraiffa1976} for a classical exposition) and two
inference methods that are based on this model.

\subsection{Additive utility function models}

Let $\succsim$ denote the preference relation of a DM on a set of alternatives.
We assume that each of these alternatives is fully described by a $n$-dimensional vector the components of which are the evaluations of the alternative w.r.t.{} $n$ criteria or attributes. Under some conditions, among which preferential independence (see \cite{keeneyraiffa1976}, p.110), such a preference can be represented by means of an additive value function. To be more precise, let $a$ (resp. $b$) denote an alternative described by the vector $(a_1, \ldots, a_n)$ (resp. $(b_1, \ldots, b_n)$) of its evaluations on $n$ criteria. The preference of the DM is representable by an additive value function if there is a function $U$ which associates a value (or score) to each alternative in such a way that $U(a)\geq U(b)$ whenever the DM prefers $a$ to $b$ ($a \succsim b$) and
\begin{equation}\label{eq:addUtil}
    U(a) = \sum_{j=1}^n w_j u_j(a_j),
\end{equation}
where $u_j$ is a marginal value function defined on the scale or range of criterion $j$ and $w_j$ is a weight or tradeoff associated to criterion $j$. Weights can be normalized w.l.o.g., i.e.\@ $\sum_{j=1}^n w_j = 1$.

In the sequel, we assume that the range of each criterion $j$ is an interval $[v_{1,j}, v_{2,j}]$ of an ordered set, e.g.\@ the real line. We assume w.l.o.g.\ that, along each criterion, the DM's preference increases with the evaluation (the larger the better). We also assume that the marginal value functions are normalized, i.e.\  $u_j(v_{1,j}) = 0$ for all $j$ and $\sum_{j=1}^n u(v_{2,j}) =1$.

Model \eqref{eq:addUtil} can be rewritten by integrating the weights in the
marginal value functions as follows: 	$u^*_j(a_j) = w_j \cdot u_j(a_j) \text{ for all } j \in N = \{1, \ldots, n\}.$

Equation \eqref{eq:addUtil} can then be reformulated as follows:
\begin{equation}
	U(a) = \sum_{j = 1}^n u^*_j(a_j).
	\label{eq-global_utility2}
\end{equation}
The marginal value functions, or, more briefly, the marginals $u^*_j$ take their values in the interval $[0,w_j]$, for all $j \in N$.
Note that a preference $\succsim$ that can be represented by a value function is necessarily a weak order, i.e.\ a transitive and complete relation. Such a relation is also called a ranking (ties are allowed).

\subsection{UTA methods for ranking and sorting problems}\label{ssec:utaMethods}

The \uta{} method was originally designed \cite{jaquetlsiskos1982} to learn the preference relation of the DM on the basis of partial knowledge of this preference. It is supposed that the DM is able to rank some pairs of alternatives \textit{a priori}, without further analysis. Assuming that the DM's preference on the set of all alternatives is a ranking which is representable by an additive value function, \uta{} is a method for learning one such function which is compatible with the DM's a priori ranking of certain pairs of alternatives.

Let $\mathcal{P}$ denote the set of pairs of alternatives
$(a,b)$ such that the DM knows a priori that he/she strictly prefers $a$ to $b$.
More precisely, if $(a,b) \in \mathcal{P}$, we have $a\succ b$, which means $a \succsim b$ and not $[b \succsim a]$. The DM may also know that he/she is indifferent between some pairs of alternatives. These constitute the set $\mathcal{I}$. Whenever $(a,b) \in \mathcal{I}$, we have $a \sim b$, i.e.\ $a \succsim b$ and $b \succsim a$.
We denote by $A^*$ the set containing the learning alternatives, i.e.\@ these used for the comparisons in sets $\mathcal{P}$ and $\mathcal{I}$.
These two sets and the vectors of performances of the alternatives
contained in these two sets constitute the learning set which serves as input  to the
learning algorithm.

Linear programming is used to infer the parameters of the \uta{} model.
Each pairwise comparison of the set $\mathcal{P}$ and $\mathcal{I}$ is
translated into a constraint.
For each pair of alternatives $(a,b) \in \mathcal{P}$, we have $U(a) -
U(b) > 0$ and for each pair of alternatives $(a,b) \in \mathcal{I}$, we
have $U(a) - U(b) = 0$.
Note that these constraints may prove incompatible.
In order to have a feasible linear program in all cases, two positive slack variable,
$\sigma^+(a)$ and $\sigma^-(a)$, are introduced for each alternative in $A^*$.
The objective function of \uta{} is given by:
\begin{equation}
	\min_{u^*_j}{\sum_{a \in A^*}
		\left( \sigma^+(a) + \sigma^-(a) \right)}
\label{eq-uta_objective}
\end{equation}
and the constraints by:
\begin{equation}
\left\{
\begin{array}{rclr}
	U(a) - U(b)
		+ \sigma^+(a) - \sigma^-(a)
		- \sigma^+(b) + \sigma^-(b) & > & 0
		& \forall (a,b) \in \mathcal{P},\\
	U(a) - U(b)
		+ \sigma^+(a) - \sigma^-(a)
		- \sigma^+(b) + \sigma^-(b) & = & 0
		& \forall (a,b) \in \mathcal{I},\\
	\sum_{j = 1}^n u^*_j(v_{2,j}) & = & 1,\\
	u^*_j(v_{1,j}) & = & 0 & \forall j \in N,\\

	\sigma^+(a) & \geq & 0 & \forall a \in A^*,\\
	\sigma^-(a) & \geq & 0 & \forall a \in A^*,\\
	u^*_j & \multicolumn{2}{l}{\text{monotonic}} & \forall j \in N.\\
\end{array}
\right.
\label{eq-uta_constraints}
\end{equation}

If we assume that the unknown marginals $u^*_j$ are piecewise linear, all the constraints above can be formulated in linear fashion and the corresponding optimization program can be handled by a LP solver. Note that the range $[v_{1,j}, v_{2,j}]$ of each criterion $j$ has to be split in a number of segments that have to be fixed a priori (i.e.\@ they are not variables in the program).
%The segment are equally spaced on the criterion domain.
%The breakpoints are denoted by $g_j^l$ with $g_j^0 = v_{1,j}$ and
%$g_j^k = v_{2,j}$.
%The value of $g_j^l$ is given by $g_j^l = g_j^0 + \frac{l}{k} \cdot (g_j^k
%- g_j^0)$.
%An example of piecewise linear utility function is shown in Figure
%\ref{fig-marginal_utility_function_plinear}.
%\begin{figure}[h]
%	\centering
%	\includegraphics{figs/marginal_utility_function_plinear.pdf}
%	\caption{Example of piecewise linear marginal utility function for
%	a criterion~$j$.}
%	\label{fig-marginal_utility_function_plinear}
%\end{figure}
%
%Using piecewise linear functions, the marginal utility of an alternative
%$a$ on criterion $j$ can be expressed as follows.
%For $l = \{1, \ldots, k\}$, let $g_j^L$ be the first breakpoint for an
%alternative $a$ such that $a_j \leq g_j^L$.
%The value $u^*_j(a_j)$ is given by:
%\begin{equation}
%	u^*_j(a_j) = u^*_j(g_j^{L - 1})
%		+ \left(\frac{a_j - g_j^{L-1}}
%		{g_j^{L} - g_j^{L - 1}}\right)
%		\left( u^*_j(g_j^L) - u^*_j(g_j^{L-1}) \right).
%	\label{eq-marginal_utility}
%\end{equation}
%Equation \eqref{eq-marginal_utility} is linear if the breakpoints
%$g_j^{L}$ and $g_j^{L-1}$ are fixed at given position on every criterion
%$j$.
%In original \uta{} and \utadis{} methods, these breakpoints are equally
%spaced on the criterion domain.

%\subsection{Sorting problem}
A variant of \uta{} for learning to sort alternatives in ordered categories is known as \utadis{}. The idea was formulated in the initial paper \cite{jaquetlsiskos1982} and further used and developed in \cite{doumposzopounidis2002,zopounidisdoumpos99}.
Let $C_1, \ldots, C_p$ denote the categories. They are numbered in increasing order of preference, i.e., an alternative assigned to $C_h$ is preferred to any alternative assigned to $C_{h'}$ for $1 \leq h' < h \leq p$. It is assumed that the alternatives assignment is compatible with the dominance relation, i.e., an alternative which is at least as good as another on all criteria is not assigned to a lower category. The learning set consists of a subset of alternatives of which the assignment to one of the categories is known (or the DM is able to assign these alternatives a priori). The problem is to learn an additive value function $U$ and $p-1$ thresholds $U_1, \ldots, U_{p-1}$ such that alternative $a$ is assigned to category $C_h$ if $U_{h-1} \leq U(a) < U_h$ for $h=1$ to $p$ (setting $U_0$ to 0 and $U_p$ to infinity, i.e. a sufficiently large value). A mathematical programming formulation of this problem is easily obtained by substituting the first two lines of \eqref{eq-uta_constraints} by the following three sets of constraints:
\begin{equation}
\left\{
\begin{array}{rclr}
	U(a) + \sigma^+(a) & \geq & U^{h - 1}
		& \forall a \in A^{*h}, h = \{2, ..., p\},\\
	U(a) - \sigma^-(a) & < & U^h
		& \forall a \in A^{*h}, h = \{1, ..., p-1\},\\

	U^h & \geq & U^{h-1} & h = \{2, ..., p-1\},\\
%
%	\sum_{j = 1}^n u^*_j(v_{2,j}) & = & 1,\\
%	u^*_j(v_{1,j}) & = & 0 & \forall j \in N,\\
%
%	U^h & \in & [0,1] & h = \{2, ..., p-1\},\\
%	\sigma^+(a) & \geq & 0 & \forall a \in A^*,\\
%	\sigma^-(a) & \geq & 0 & \forall a \in A^*,\\
%	u^*_j & \multicolumn{2}{l}{\text{monotonic}} & \forall j \in N.\\
\end{array}
\right.
\label{eq-utadis_constraints}
\end{equation}
where $A^{*h}$ denotes the alternatives in the learning set that are assigned to category $C_h$. Assuming that marginals are piecewise linear, allows for a linear programming formulation as it is the case with \uta{}.
%As for \utadis{}, in the traditional \uta{} methods, the monotonicity
%constraints of marginal utility functions, $u^*_j$ is guaranteed by using
%piecewise linear functions.
%The breakpoints of the functions are usually equidistant on the criterion
%domain.

%% file: utapRev.tex
% vim:spell spelllang=en

\section{\utap{}: additive value functions with polynomial marginals}

In this section we present a new way to elicit marginal value functions
using semidefinite programming.
We first give the motivations for this new method.
Then we describe it.

\subsection{Motivation}

\uta{} methods use piecewise linear functions to model the marginal
value functions.
Opting for such functions allows to use the linear programs presented in the previous section and linear
programming solvers to infer an additive value ranking or sorting model.
However by considering piecewise linear marginals with breakpoints at predefined places, original \uta{}
methods have two important drawbacks: these options limit the interpretability and flexibility of the additive value model.

\smallskip \noindent \textit{Interpretability.} There is a longstanding tradition in Economics, especially in the classical theory of consumer behavior (see e.g.\ \cite{SilberbergSuen2001}), which assumes that utility (or value) functions are differentiable and interpret their first and second (partial) derivatives in relation with the preferences and behavior of the customer. Multiple criteria decision analysis, based on value functions, stems from the same tradition. Tradeoffs or marginal rates of substitution are generally thought of as changing smoothly (see e.g.\ \cite{keeneyraiffa1976}, p. 83 :``Throughout we assume that we are in a well-behaved world where all functions have smooth second derivatives''). Although piecewise linear marginals can provide good approximations for the value of any derivable function, they are not fully satisfactory as an explanatory model. This is especially the case when the breakpoints are fixed arbitrarily (e.g.\ equally spaced in the criterion domains). Such a choice may well fail to correctly reflect the DM's feelings about where the marginal rate of substitution starts to grow more quickly (resp. to diminish) or shows an inflexion. In other words, the qualitative behavior of the first and second derivatives of the ``true'' marginal value function might be poorly approximated by resorting to piecewise linear models, while this behavior might have an intuitive meaning for the DM. Therefore, considering piecewise linear marginals might lead to final models that fail to convince the DM even though they fit the learning set accurately.

%Indeed, the positions of the breakpoints determine the places where the
%slope and the concavity of the marginal value function may change.
%It implies two limitations.
%Firstly, it is not guaranteed that the pre-defined positions are the ones
%where the intensity of the marginal (preference) varies.
%Indeed, in the (real) unveiled preferences of a \dm{}, the variation of
%the slope may occur before or after the fixed breakpoints.
%Secondly, the concavity of the curve may also vary at other places.
%The lack of flexibility of \uta{} methods may results in not well chosen
%marginals and may decrease the accuracy of the method.
%
%Moreover, in \uta{} the form of the value function close to the
%breakpoints is not natural.
%Indeed, there is no continuity of the derivatives at the breakpoints.
%The slope of marginal may change abruptly at these points.

\smallskip
\noindent \textit{Flexibility.} Restricting the shape of the marginals to piecewise linear functions with a fixed number of pieces may hamper the expressivity of the additive value function model. This is especially detrimental when large learning sets are available as is the case in Machine Learning applications\footnote{It is seldom so in \mcda{} applications where the size of the learning set rarely exceeds a few dozens records.}.

The following \textit{ad hoc} case aims to illustrate the loss in flexibility incurred due to the piecewise linear hypothesis. We hereafter illustrate the case of a single piece, i.e.\ the linear case, whereas the same question arises whatever the fixed number of segments.
Consider a ranking problem in which alternatives are assessed on two criteria.
The \dm{} states that the top-ranked alternatives are $a$, $b$, which are tied (rank 1), followed by $c$ (rank 2) while $d$ is strictly less preferred than the others (rank 3).
The evaluations and ranks of these alternatives are displayed in Table \ref{table-example_utap_flexibility}.

\begin{table}
\centering
\begin{tabular}{rrrr}
	\toprule
	alternative & criterion 1 & criterion 2& rank\\
	\midrule
	$a$ & 100 & 0 & $1$\\
	$b$ & 0 & 100 & $1$\\
	$c$ & 25 & 75 & $2$\\
	$d$ & 75 & 25 & $3$\\
	\bottomrule
\end{tabular}
\caption{Example of an alternatives ranking that is not representable
with a \uta{} model (one linear piece per marginal).}
\label{table-example_utap_flexibility}
\end{table}

Assume that we plan to use a \uta{} model with marginals involving a single linear piece (i.e.\ a weighted sum).
Such an \uta{} model cannot at the same time distinguish $c$ and $d$ and express that $a$ and $b$ are tied.
The fact that $a$ and $b$ are tied indeed implies that the criteria weights are equal (we can set them to $0.5$ w.l.o.g.).
The value on each marginal varies from 0 to 0.5.
The worst value (0) corresponds to the worst performance (0) and the best value (0.5) to the best performance (100) on each
criterion (see the marginal value functions represented by dashed lines in Figure
\ref{fig-example_utap_vs_uta}).
Using these marginals, the scores of the four alternatives are
obtained through linear interpolation and displayed in Table \ref{table-scores_uta_utap}.
We observe that all alternatives receive the same value 0.5.
It is therefore not possible to discriminate alternatives $c$ and $d$ without increasing the number of linear pieces or considering nonlinear marginals.
In this case, we shall consider using non-linear marginals.
\begin{figure}
	\centering
	\includegraphics{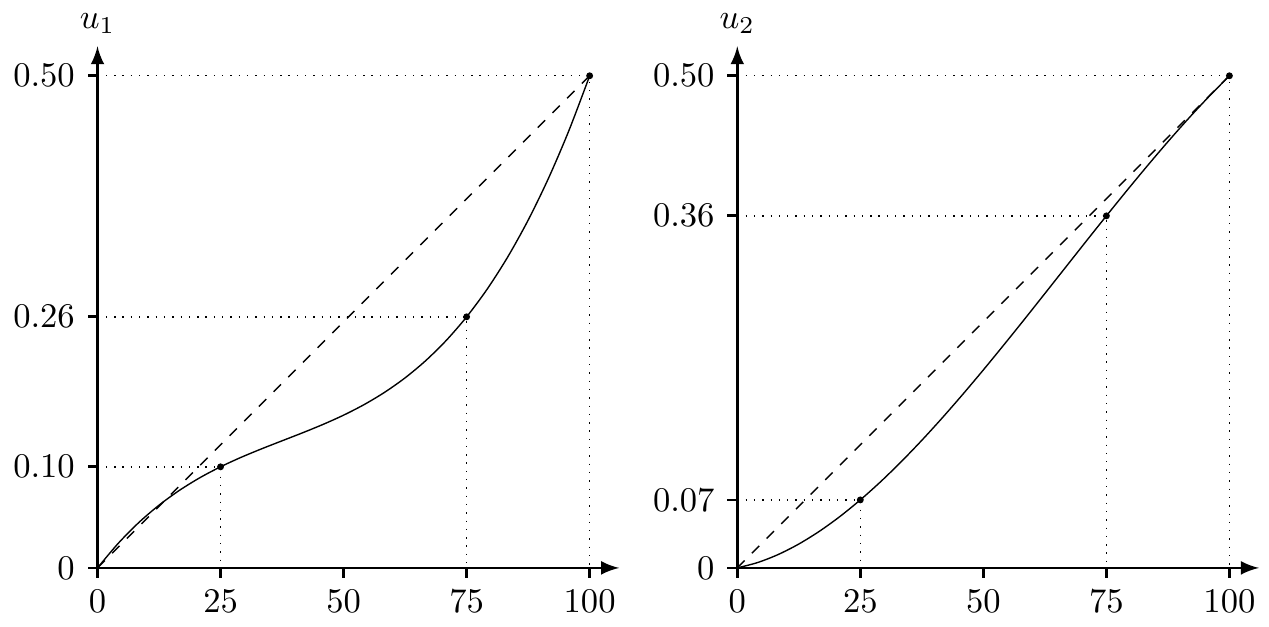}
	\caption{Example of \uta{} and \utap{} value functions.
	The dashed lines correspond to the \uta{} piecewise linear
	function and the plain lines correspond to polynomials of
	 degree 3.}
	\label{fig-example_utap_vs_uta}
\end{figure}

\begin{table}
\centering
\begin{tabular}{rrrrr}
	\toprule
	& $a$ & $b$ & $c$ & $d$\\
	\midrule
	\uta{} score & 0.5 & 0.5 & 0.5 & 0.5\\
	\utap{} score & 0.5 & 0.5 & 0.46 & 0.33\\
	\bottomrule
\end{tabular}
\caption{\uta{} and \utap{} scores of the alternatives described in Table
\ref{table-example_utap_flexibility} with the \uta{} and \utap{} marginals
represented in Figure \ref{fig-example_utap_vs_uta}.}
\label{table-scores_uta_utap}
\end{table}

In case polynomials are allowed for, instead of piecewise linear
functions, to model the marginals, the DM's preferences can be accurately represented.
Figure \ref{fig-example_utap_vs_uta} shows the case of polynomials of degree 3 used as
marginals (plain line). The scores of the alternatives computed with these marginals are displayed in Table \ref{table-scores_uta_utap}. They comply with  the
\dm{}'s preferences.

Obviously it would have been possible to reproduce the \dm{}'s ranking using more than one linear piece marginals in an \uta{} model.
However, when the breakpoints are fixed in advance, it is easy to construct an example, similar to the above one, in which the \dm{}'s ranking cannot be reproduced using a linear function between successive breakpoints while a polynomial spline will do.

The two methods introduced below, \utap{} in the rest of this section and \utasp{} in Section \ref{sec:utaspline}, replace the piecewise linear marginals of \uta{} by polynomials and polyomial splines, respectively.
%It enables to compute the derivative of the function in any point on the
%criterion domain.
%A difficulty arising when using polynomials as marginal value functions
%is to guarantee its monotonicity.

\subsection{Basic facts about non-negative polynomials}
\label{subsection-semidefinite_programming}

In the last few years, significant improvements have been made in formulating and solving optimization problems in which
constraints are expressed in the form of polynomial (in)equalities and
with a polynomial objective function; see, e.g., \cite{gloptipoly, gloptipoly3}.
These new techniques are useful for various applications; see~\cite{Las09} and the references therein.
A problem arising in many applications, including the present one, is to guarantee the
non-negativity of functions of several variables. In our case, we have to make sure not only that marginals are non negative but also that they are nondecreasing, i.e. that their derivative is non-negative.
Testing the non-negativity of a polynomial of several variables and of a degree equal to or greater
than 4 is NP-hard \cite{murtykabadi1987}.
In \cite{parillo2003}, an approach based on convex optimization techniques
has been proposed in order to find an approximate solution to this
problem.

The approach proposed in \cite{parillo2003} is based on the following
theorem about non-negative polynomials.
\begin{theorem}[Hilbert]
A polynomial $F: \mathbb{R}^n \to \mathbb{R}$ is non-negative if it is
possible to decompose it as a sum of squares (SOS):
\begin{equation}
	F(z) = \sum_s f_s^2(z) \qquad
		\text{with } z \in \mathbb{R}^n.
	\label{eq-poly_sos}
\end{equation}
\end{theorem}
The condition given above is sufficient but not necessary, there exist
non-negative polynomials that cannot be decomposed as a sum of squares
\cite{blekherman2006}.
However, it has been proved by Hilbert that a non-negative polynomial
of one variable is always a sum of squares \cite{parillo2003}.
We give the proof here because it is remarkably simple and elegant.
\begin{theorem}[Hilbert]
	A non-negative polynomial in one variable is always a SOS.
\end{theorem}
\begin{proof}
Consider a polynomial of degree $D$, $p(x) = p_0 + p_1 x + p_2 x^2 +
\ldots + p_D x^D$. Since $p(x)$ is non-negative, $D$ must be even.
The value of $p_D$ should be greater than 0, otherwise $\lim_{x \to
\infty} p(x) = - \infty$.
As every polynomial of degree $D$ admits $D$ roots, one can write $p(x)$ as
follows: $$p(x) = p_D \prod_{i = 1}^m (x - z_i) (x - \bar{z}_i)
\prod_{j = 1}^n (x - t_j)^{\alpha_j}$$ in which $z_i$ and
$\bar{z}_i$ for $i = \{1, \ldots, m\}$ are pairs of conjugate complex
numbers and $t_j$ for $j = \{1, \ldots, n\}$ are distinct real numbers where $D =
2m + \sum_{j=1}^n \alpha_j$.
All the values of the exponents $\alpha_j$ are even.
Indeed, consider a subset of $k$ indices, $\{\Delta_1, \ldots, \Delta_k\}$,
such that $\alpha_{\Delta_1}, \ldots, \alpha_{\Delta_k}$ are odd.
Let $\tau$ be a permutation of these indices such that $t_{\tau(\Delta_1)}
< \ldots < t_{\tau(\Delta_k)}$.
For $x \in \left]t_{\tau(\Delta_{k-1})},
t_{\tau(\Delta_k)}\right[$, we would have $\prod_{j=1}^n (x -
t_j)^{\alpha_j} < 0$, a contradiction.
As all the value $\alpha_j$ are even, we can rewrite $p(x)$ as follows:
$$p(x) = \left(\sqrt{p_D} \prod_{i = 1}^l (x - z_i) \right)
\left(\sqrt{p_D} \prod_{i = 1}^l (x - \bar{z}_i) \right)$$
in which some pairs $(z_i$, $\bar{z}_i)$ have no imaginary part.
Let $\left(\sqrt{p_D} \prod_{i = 1}^l (x - z_i) \right) = q(x) + i r(x)$
and $\left(\sqrt{p_D} \prod_{i = 1}^l (x - \bar{z}_i) \right) = q(x) - i
r(x)$ where $i$ is the imaginary part of the complex number and $q(x),
r(x)$, two polynomials with real coefficients.
Finally, the product of these two terms gives a sum of two squares:
$p(x) = \left[q(x)\right]^2 + \left[r(x)\right]^2$.
\end{proof}

Let us consider the problem of determining a non-negative polynomial $p$ of
one variable $x$ and degree $D$.
We use the following canonical form to represent this polynomial:
\begin{align}
	p(x)
	& = p_0 + p_1 x + p_2 x^2 + \ldots + p_D x^D \label{eq-poly_cano}\\
	& = \sum_{i = 0}^D p_i \cdot x^i. \nonumber
\end{align}
To guarantee the non-negativity of this polynomial, we have to ensure that it
can be represented as a sum of squares like in Equation \eqref{eq-poly_sos}.
Note that a non-negative polynomial will always have an even degree since
either the limit at positive or negative infinity of a polynomial of odd
degree is negative.
Let $d = \frac{D}{2}$, the polynomial $p(x)$ reads:
\begin{align*}
	p(x) = \sum_s q_s^2(x)
	= \sum_s \left[ \sum_{i = 0}^d b_s^i x^i \right]^2.
\end{align*}
Defining $b_s^\transp = \begin{pmatrix} b_s^0 & b_s^1 & \ldots & b_s^d
\end{pmatrix}$ and
$\overline{x}^T = \begin{pmatrix} 1 & x & \ldots & x^d \end{pmatrix}$ (where $\transp$ stands for the matrix transposition operation), we can express $p(x)$
as follows:
\begin{align*}
	p(x)
	& = \sum_s \left( b_s^\transp \overline{x} \right)^2
	= \sum_s \overline{x}^\transp b_s  b_s^\transp
		\overline{x} \nonumber
	= \overline{x}^\transp \left[\sum_s b_s b_s^\transp \right]
		\overline{x} \nonumber
			= \overline{x}^\transp Q \overline{x} \nonumber\\
	& =
	\begin{pmatrix}
		1\\
		x\\
		\vdots\\
		x^d
	\end{pmatrix}^\transp
	\begin{pmatrix}
	q_{0,0} & q_{0,1} & \cdots & q_{0,d}\\
	q_{1,0} & q_{1,1} & \cdots & q_{1,d}\\
	\vdots & \vdots & \ddots & \vdots\\
	q_{d,0} & q_{d,1} & \cdots & q_{d,d}\\
	\end{pmatrix}
	\begin{pmatrix}
		1\\
		x\\
		\vdots\\
		x^d
	\end{pmatrix}.
\end{align*}
Note that the matrix $Q = \sum_s b_s b_s^\transp$ is symmetric and positive semidefinite (PSD),
which we denote $Q \succeq 0$,
since $\overline{x}^\transp Q \overline{x} =  \sum_s \left( b_s^\transp \overline{x} \right)^2 \geq 0$ for all $\overline{x} \in \mathbb{R}^{d+1}$.
Therefore, to ensure that $p(x)$ is non-negative, it is necessary to find a matrix
$Q$ of dimension $(d + 1) \times (d + 1)$ such that $p(x) =
\overline{x}^\transp Q \overline{x}$ and $Q \succeq 0$.
It turns out that this condition is also sufficient. This follows from the following lemma.
\begin{lemma}
	$Q \succeq 0 \iff \exists H : Q = H \cdot H^\transp$.
\end{lemma}
The above decomposition is called the Cholesky decomposition of matrix
$Q$; see \ref{section-cholesky}.
To summarize, \emph{a polynomial $p(x)$ in one variable is non-negative if
and only if there exists $Q \succeq 0$ such that $p(x) =
\overline{x}^\transp Q \overline{x}$}.

The coefficients of the polynomial expressed in its canonical form
\eqref{eq-poly_cano} are obtained by summing the
off-diagonal entries of the matrix $Q$, as follows:
\begin{equation*}
\begin{cases}
	p_0 = q_{0,0},\\
	p_1 = q_{1,0} + q_{0,1},\\
	p_2 = q_{2,0} + q_{1,1} + q_{0,2},\\
	\vdots\\
	p_{2d - 1} = q_{d,d-1} + q_{d-1,d},\\
	p_{2d} = q_{d,d}.
\end{cases}
\end{equation*}
We can express the value of the coefficients of the polynomial as follows:
\begin{equation}
	p_i =
	\begin{cases}
		\sum_{g = 0}^i q_{g,i-g} & i = \{0, \ldots, d\},\\
		\sum_{g = i - d}^d q_{g,i-g} & i = \{d, \ldots, 2d\}.
	\end{cases}
\end{equation}
The value of $p_d$ can be computed with both expressions.
Finding a non-negative univariate polynomial consists in finding a
semidefinite positive matrix $Q$.
Summing the off-diagonal entries of this matrix allows to control the
coefficients of the polynomial;

In some applications, it is not necessary to ensure the non-negativity of the
polynomial on $\mathbb{R}$ but only in an interval $[v_1, v_2]$.
If the non-negativity constraint has to be guaranteed only in a given
interval $[v_1, v_2]$ for a polynomial $p(x)$, then the following theorem
holds.
\begin{theorem}[Hilbert]
\label{theorem-sos_interval}
A polynomial $p(x)$ in one variable $x$ is non-negative in the interval
$[v_1, v_2]$, if and only if
$p(x) = (x - v_1) \cdot q(x) + (v_2 - x) \cdot r(x)$ where $q(x)$ and
$r(x)$ are SOS.
\end{theorem}

Given the above theorem, if we want to ensure the non-negativity of the
polynomial $p(x)$ of degree $D$ on the interval $[v_1, v_2]$, we have to
find two matrices $Q$ and $R$ of size $d + 1$, with $d = \left\lfloor
\frac{D}{2} \right\rfloor$, that are positive semidefinite.
We denote these matrices and their indices as follows:
\begin{equation*}
	Q =
	\begin{pmatrix}
	q_{0,0} & q_{0,1} & \cdots & q_{0,d}\\
	q_{1,0} & q_{1,1} & \cdots & q_{1,d}\\
	\vdots & \vdots & \ddots & \vdots\\
	q_{d,0} & q_{d,1} & \cdots & q_{d,d}\\
	\end{pmatrix},\quad
	R =
	\begin{pmatrix}
	r_{0,0} & r_{0,1} & \cdots & r_{0,d}\\
	r_{1,0} & r_{1,1} & \cdots & r_{1,d}\\
	\vdots & \vdots & \ddots & \vdots\\
	r_{d,0} & r_{d,1} & \cdots & r_{d,d}\\
	\end{pmatrix}.
\end{equation*}
Since $Q$ and $R$ are positive semidefinite, the products
$\overline{a_j}^\transp Q \overline{a_j}$ and $\overline{a_j}^\transp R
\overline{a_j}$, with $\overline{a_j}^\transp = \begin{pmatrix} 1 & a_j &
\ldots & a_j^d\end{pmatrix}$, are always non-negative.

To obtain a polynomial $p(x)$ that is non-negative in the interval $[v_1,
v_2]$, its coefficients have to be chosen such that:
\begin{equation*}
\begin{cases}
	p_0 = v_2 \cdot r_{0,0} - v_1 \cdot q_{0,0},\\
	p_1 = q_{0,0} - r_{0,0}
		+ v_2 \cdot (r_{1,0} + r_{0,1})
		- v_1 \cdot (q_{1,0} - q_{0,1}),\\
	p_2 = (q_{1,0} + q_{0,1}) - (r_{1,0} + r_{0,1})
		+ v_2 \cdot (r_{2,0} + r_{1,1} + r_{0,2})\\
	\qquad \qquad - v_1 \cdot (q_{2,0} + q_{1,1} + q_{0,2}),\\
	\vdots\\
	p_{2d-1} = (q_{d,d-2} + q_{d-1,d-1} + q_{d-2,d})
		- (r_{d,d-2} + r_{d-1,d-1} + r_{d-2,d})\\
	\qquad \qquad + v_2 \cdot (r_{d,d-1} + r_{d-1,d})
		- v_1 \cdot (q_{d,d-1} + q_{d-1,d}),\\
	p_{2d} = (q_{d,d-1} + q_{d-1,d}) - (r_{d,d-1} + r_{d-1,d})
		+ v_2 \cdot r_{d,d} - v_1 \cdot q_{d,d},\\
	p_{2d+1} = q_{d,d} -  r_{d,d}.
\end{cases}
\end{equation*}
If the degree $D$ of the polynomial $p(x)$ is even then the value of
$p_{2d+1}$ is equal to 0.
The values $p_i$ can be expressed in the following more compact form:
\begin{equation*}
	p_i =
	\begin{cases}
	v_2 \cdot r_{0,0} - v_1 \cdot q_{0,0} & i = 0,\\
	\sum_{g = 0}^{i-1} (q_{g,i-1-g} - r_{g, i-1-g})\\
		\qquad + \sum_{g = 0}^i (v_2 \cdot r_{g, i-g}
			- v_1 \cdot q_{g,i-g})
		& i = \{1, \ldots, d\},\\
	\sum_{g = i - d - 1}^d (q_{g,i-1-g} - r_{g,i-1-g})\\
		\qquad + \sum_{g=i-d}^d (v_2 \cdot r_{g,i-g}
			- v_1 \cdot q_{g,i-g})
		& i = \{d+1, \ldots, 2d\},\\
	q_{d,d} - r_{d,d} & i = 2d + 1.
	\end{cases}
\end{equation*}

\subsection{Semidefinite programming applied to \uta{} methods}

In the perspective of building more natural marginal value functions, we
use semidefinite programming (SDP) to learn polynomial marginals
 instead of piecewise linear ones. SDP has become a
standard tool in convex optimization, being a generalization of linear
programming and second-order cone programming. It allows to optimize
linear functions over an affine subspace of the set of positive
semidefinite matrices; see, e.g., \cite{vdb1996sdp} and the references
therein.

There are two variants of the new \utap{} method.
Firstly, we describe the approach that consists in using polynomials that
are overall monotone, i.e.\@ monotone on the set of all real numbers.
Then we describe the second approach considering polynomials that are
monotone only on a given interval.

\subsubsection{Enforcing monotonicity of the marginals on the set of real numbers}

In the new proposed model, we define the value function on each
criterion $j$ as a polynomial of degree $D_j$:
\begin{equation}
	u_j^*(a_j) = \sum_{i = 0}^{D_j} p_{j,i} \cdot a_j^i.
	\label{eq-utilities_polynomials}
\end{equation}
To be compliant with the requirements of the theory of additive value functions, the polynomials used as marginals should be non-negative and monotone over the criteria domains.
To ensure monotonicity, the derivative of the marginal value function has to be
non-negative, hence we impose that the derivative of each value
function is a sum of squares.
The degree of the derivative is therefore even which implies that $D_j$ is
odd.
This requirement reads:
\begin{align*}
	u_j^*{'}
	& = p_{j,1} + 2 p_{j,2} \cdot a_j + 3 p_{j,3} \cdot a_j^2 + ...
		+ D_j p_{j,D_j} \cdot a_j^{D_j-1}\\
	& = \overline{a_j}^\transp Q_j \overline{a_j},
\end{align*}
with $Q_j$ a PSD matrix of dimension $(d_j + 1) \times (d_j + 1)$,
$\overline{a_j}$ a vector of size $(d_j + 1)$ with $d_j =
\frac{D_j - 1}{2}$:
\begin{equation*}
	Q_j =
	\begin{pmatrix}
	q_{j,0,0} & q_{j,0,1} & \cdots & q_{j,0,d_j}\\
	q_{j,1,0} & q_{j,1,1} & \cdots & q_{j,1,d_j}\\
	\vdots & \vdots & \ddots & \vdots\\
	q_{j,d_j,0} & q_{j,d_j,1} & \cdots & q_{j,d_j,d_j}\\
	\end{pmatrix},
	\quad
	\overline{a_j} =
	\begin{pmatrix}
	1\\
	a_j\\
	\vdots\\
	a_j^{d_j}
	\end{pmatrix}.
\end{equation*}

By using SDP, we impose the matrix $Q$ to be semidefinite
positive and we set the following constraints on the $p_{j,i}$ values, for
$i \geq 1$:
\begin{equation*}
\begin{cases}
	p_{j,1} = q_{j,0,0},\\
	2 p_{j,2} = q_{j,1,0} + q_{j,0,1},\\
	3 p_{j,3} = q_{j,2,0} + q_{j,1,1} + q_{j,0,2},\\
	\vdots\\
	(2d_j) p_{j,2d_j} = q_{j,d_j,d_j-1} + q_{j,d_j-1,d_j},\\
	(2d_j + 1) p_{j,2d_j+1} = q_{j,d_j,d_j}.
\end{cases}
\end{equation*}

In \utap{}, the marginal value functions and monotonicity conditions on
marginals given in Equation \eqref{eq-uta_constraints} and
\eqref{eq-utadis_constraints} are replaced by the following constraints:
\begin{equation}
\left\{
\begin{array}{rclr}
	U(a) & = & \sum_{j = 0}^n \sum_{i = 0}^{D_j} p_{j,i} \cdot a_j^i
		& \forall a \in A,\\
	Q_j & \multicolumn{2}{l}{\text{PSD}} & \forall j \in N,\\
	(i + 1) p_{j,i+1} & = & \sum_{g = 0}^i q_{j,g,i-g}
		& i = \{ 0, \ldots, d_j \}, \forall j \in N,\\
	(i + 1) p_{j,i+1} & = & \sum_{g = i-d_j}^{d_j} q_{j,g,i-g}
		& i = \{ d_j+1, \ldots, 2 d_j \}, \forall j \in N.
\end{array}
\right.
\label{eq-utap_constraints_monotonicity}
\end{equation}

The optimization program composed of the objective given in Equation \eqref{eq-uta_objective} and
the set of constraints given in Equations \eqref{eq-uta_constraints} and
\eqref{eq-utap_constraints_monotonicity} can be solved using convex
programming, more precisely, semidefinite programming \cite{parillo2003}.
We refer to this new mathematical program as to \utap{}.
An explicit \utap{} formulation for a simple problem involving 2 criteria and 3 alternatives is provided in
\ref{appendix-example_sdp} for illustrative purposes.

\subsubsection{Enforcing monotonicity of the marginals on the criteria domains}

Ensuring the monotonicity of each marginal on the
domain of each criterion (instead of the whole real line) is sufficient to satisfy the requirements of the additive value function
model. To do so, we use Theorem \ref{theorem-sos_interval} and only
impose  the non-negativity of the marginal derivative on the domain $[v_{1,j}, v_{2,j}]$
of each criterion. This results in the following condition on the derivative $u_j^*{'}$ of the polynomial $u_j^*$, for all $j$:
\begin{align*}
	u_j^*{'}
	& = p_{j,1} + 2 p_{j,2} \cdot a_j + 3 p_{j,3} \cdot a_j^2 + ...
		+ D_j p_{j,D_j} \cdot a_j^{D_j-1}\\
	& = (a_j - v_{1,j}) \overline{a_j}^\transp Q_j \overline{a_j}
		+ (v_{2,j} - a_j) \overline{a_j}^\transp R_j \overline{a_j}.
\end{align*}
In the above equation, $Q_j$ and $R_j$ are two PSD matrices of size
$(d_j + 1) \times (d_j + 1)$ and $\overline{a_j}$ a vector of size $d_j +
1$, where $d_j = \left\lfloor \frac{D_j - 1}{2} \right\rfloor$:
\begin{equation*}
	Q_j =
	\begin{pmatrix}
	q_{j,0,0} & q_{j,0,1} & \cdots & q_{j,0,d_j}\\
	q_{j,1,0} & q_{j,1,1} & \cdots & q_{j,1,d_j}\\
	\vdots & \vdots & \ddots & \vdots\\
	q_{j,d_j,0} & q_{j,d_j,1} & \cdots & q_{j,d_j,d_j}\\
	\end{pmatrix},
	\quad
	R_j =
	\begin{pmatrix}
	r_{j,0,0} & r_{j,0,1} & \cdots & r_{j,0,d_j}\\
	r_{j,1,0} & r_{j,1,1} & \cdots & r_{j,1,d_j}\\
	\vdots & \vdots & \ddots & \vdots\\
	r_{j,d_j,0} & r_{j,d_j,1} & \cdots & r_{j,d_j,d_j}\\
	\end{pmatrix}.
\end{equation*}
The value $p_{j,i}$ for $i \geq 1$ are obtained as follows:
\begin{equation*}
\begin{cases}
	p_{j,1} = v_{2,j} \cdot r_{j,0,0} - v_{1,j} \cdot q_{j,0,0},\\
	2p_{j,2} = q_{j,0,0} - r_{j,0,0}
		+ v_{2,j} \cdot (r_{j,1,0} + r_{j,0,1})
		- v_{1,j} \cdot (q_{j,1,0} + q_{j,0,1}),\\
	3p_{j,3} = (q_{j,1,0} + q_{j,0,1}) - (r_{j,1,0} + r_{j,0,1})
		+ v_{2,j} \cdot (r_{j,2,0} + r_{j,1,1} + r_{j,0,2})\\
	\qquad \qquad \qquad - v_{1,j} \cdot (q_{j,2,0} + q_{j,1,1} + q_{j,0,2})\\
	\vdots\\
	(2d_j) p_{j,2d_j} = (q_{j,d_j,d_j-2} + q_{j,d_j-1,d_j-1} + q_{j,d_j-2,d_j})\\
	\qquad \qquad \qquad \qquad - (r_{j,d_j,d_j-2} + r_{j,d_j-1,d_j-1} + r_{j,d_j-2,d_j})\\
	\qquad \qquad \qquad \qquad + v_{2,j} \cdot (r_{j,d_j,d_j-1} + r_{d_j-1,d_j})
		- v_{1,j} \cdot (q_{j,d_j,d_j-1} + q_{j,d_j-1,d_j}),\\
	(2d_j+1)p_{j,2d_j+1} = (q_{j,d_j,d_j-1} + q_{j,d_j-1,d_j})
		- (r_{j,d_j,d_j-1} + r_{j,d_j-1,d_j})\\
	\qquad \qquad \qquad \qquad \qquad + v_{2,j} \cdot r_{j,d_j,d_j} -
		v_{1,j} \cdot q_{j,d_j,d_j},\\
	(2d_j+2)p_{j,2d_j+2} = q_{j,d_j,d_j} - r_{j,d_j,d_j}.
\end{cases}
\end{equation*}
If the degree $D_j$ is odd, then we have $p_{j,2d_j+2} = 0$ since $2d_j + 2 >
D_j$.

In convex programming, in order to have polynomial marginals that are
monotone on an interval, the monotonicity constraints in \uta{}
have to be replaced by the following ones:
\begin{equation}
\left\{
\begin{array}{rclr}
	U(a) & = & \sum_{j = 0}^n \sum_{i = 0}^{D_j} p_{j,i} \cdot a_j^i
		& \forall a \in A,\\
	Q_j, R_j & \multicolumn{2}{l}{\text{PSD}} & \forall j \in N,\\
	p_{j,1} & = & v_{2,j} \cdot r_{j,0,0} - v_{1,j} \cdot q_{j,0,0},\\
	(i + 1)p_{j,i+1} & =
		& \sum_{g = 0}^{i-1} (q_{j,g,i-g} - r_{j,g,i-g})\\
		& & \multicolumn{2}{l}{
		\quad + \sum_{g = 0}^i (v_{2,j} \cdot r_{j,g,i-1-g}
		- v_{1,j} \cdot q_{j,g,i-1-g})}\\
		& & \multicolumn{2}{r}{i = \{0, \ldots, d_j\}, \forall j \in N,}\\
	(i + 1) p_{j,i+1} & =
		& \multicolumn{2}{l}{
		\sum_{g = i - d_j - 1}^{d_j} (q_{j,g,i-1-g} - r_{j,g,i-1-g})}\\
		& & \multicolumn{2}{l}{
		\quad + \sum_{g = i-d_j}^{d_j} (v_{2,j} \cdot r_{j,g,i-g}
		- v_{1,j} q_{j,g, i-g})}\\
		& & \multicolumn{2}{r}{i = \{d_j+1, \ldots, 2d_j\}, \forall j \in N,}\\
	(2d_j + 2) p_{j,2d_j+2} & = & q_{d_j,d_j} - r_{d_j,d_j}
		& \qquad \qquad \qquad \quad \forall j \in N.
\end{array}
\right.
\label{eq-utap_constraints_monotonicity2}
\end{equation}

The optimization program composed of the objective given in Equation \eqref{eq-uta_objective} and
the set of constraints given in Equation \eqref{eq-uta_constraints} and
\eqref{eq-utap_constraints_monotonicity} can be solved using semidefinite
programming.

%% file: utaspRev.tex
% vim:spell spelllang=en

\section{\utasp{}: additive value functions with splines marginals}\label{sec:utaspline}

In this section we describe a variant of \utap{} which consists
in using several polynomials for each value function.
We first recall some theory about splines.
Then we describe the new method called \utasp{}.

\subsection{Splines}

We recall here the definition of a spline.
We detail the ones that are the most commonly used.

\subsubsection{Definition}

A spline of degree $D_s$ is a function $\mathit{Sp}$ that interpolates the
set of points $(x_i, y_i)$ for $i = 0, ..., q$, with
$x_0 < x_1 < \ldots < x_q$ such that:
\begin{itemize}
    \item $\mathit{Sp}(x_i) = y_i$ for $i = 0, \ldots, q$;

    \item $\mathit{Sp}$ is a set of polynomials of degree equal to or
    smaller than $D_s$, on each interval $[x_i, x_{i+1}[$ (at least one
    of the polynomials has a degree equal to $D_s$);

    \item the derivative of $\mathit{Sp}$ are continuous up to a given
    degree $D_c$ on $[x_0, x_q]$.
\end{itemize}

The degree of a spline corresponds to its highest polynomial degree.
If all the polynomials have the same degree, the spline is said to be
uniform.

The continuity of the spline at the connection points is ensured up to
a given derivative.
Usually, the continuity of the spline is guaranteed up to the second
derivative ($D_c = 2$).
It ensures the continuity of the slope and concavity at the connection
points.

\subsubsection{Cubic splines}

The most common uniform splines are the ones of degree 3 ($D_s = 3$), also
called cubic splines.
A cubic spline consists of a set of third degree polynomials which are
continuous up to the second derivative at their connection points.

We denote by $s_i$ the $i^{\text{th}}$ polynomial of the spline going from
connection point $x_{i}$ to connection point $x_{i+1}$.
Formally, each polynomial $s_i$ of the spline has the following form:
\begin{equation*}
    s_i(x) = s_{i,0} + s_{i,1} x + s_{i,2} x^2 + s_{i,3} x^3.
\end{equation*}
The use of cubic splines requires the determination of four parameters:
$s_{i,0}$, $s_{i,1}$, $s_{i,2}$ and $s_{i,3}$.
If the spline interpolates $q$ points, there are overall $4 \cdot
(q - 1)$ parameters to determine.

Imposing the equality up to the second derivative at the connection points
amounts to enforce the following constraints:
\begin{equation}
\left\{
\begin{array}{rclr}
    s_i(x_i)    & = & y_i   & i = \{0, \ldots, q - 1\},\\
    s_i(x_{i+1})    & = & y_{i+1}   & i = \{0, \ldots, q - 1\},\\
    s'_i(x_{i+1}) & = & s'_{i+1}(x_{i + 1}) & i = \{0, \ldots, q - 2\},\\
    s''_i(x_{i+1}) & = & s''_{i+1}(x_{i + 1}) & i = \{0, \ldots, q - 2\}.
\end{array}
\right.
\end{equation}
Since there are $4q - 2$ constraints and $4q$ parameters, two degrees of freedom remain.
They can be set in different ways.
For instance, one can impose $s''_{0}(x_0) = 0$ and $s''_{q-1}(x_q) = 0$.
This corresponds to imposing zero curvature at both endpoints of the spline.

\subsection{\utasp{}: using splines as marginals}

We give some detail on how using splines to model marginal
value functions of an additive value function model.
We formulate a semidefinite program that learns
the parameters of such a model.

\subsubsection{Overview}

Using splines continuous up to either the first or the second derivative instead of
piecewise linear functions for the marginal value functions aims at obtaining
more natural functions around the breakpoints.

With \utap{}, the flexibility of the model is improved by using
polynomials of higher degrees.
In order to further improve the flexibility of the model, we propose now
to hybridize the original \uta{} method which splits the criterion
domain into $k$ equal parts with the \utap{} approach which uses
polynomials to model the marginal value functions.
We call this new disaggregation procedures \utasp{}.
The \utasp{} method combines the use of piecewise functions for the
marginals (as in \uta{}) and the use polynomials (as in \utap{}) for each
piece of the function.

Compared to \uta{}, in \utasp{} the continuity of the marginal can be ensured
up to the any derivative at the connection points.
It enables to obtain more natural marginals which have a continuous
curvature.

Constraints concerning the concavity/convexity of the marginal
value functions on some sub-intervals can also be specified, if
the information is available or if the decision maker is able to
specify such constraints. This makes it possible to ``control''
the shape of the obtained model and improve its interpretability
by the decision maker.

\subsubsection{Description of \utasp{}}

In \utasp{}, we model marginals as uniform splines of degree $D_s$.
Formally the marginal of criterion $j$ reads:
\begin{equation*}
    u^*_j(a_j) = \mathit{Sp}^{D_s,k}_j(a_j)
\end{equation*}
where $\mathit{Sp}^{D_s}_j$ denotes a uniform spline of degree $D_s$
composed of $k$ pieces.
Each piece of the spline $\mathit{Sp}^{D_s,k}_j(a_j)$ is a polynomial of
degree $D_s$ denoted by $s_{j,l}(a_j)$, $l = \{1, \dots, k\}$.
Formally it reads:
\begin{equation*}
    s_{j,l}(a_j) = s_{j,l,0} + s_{j,l,1} a_j + s_{j,l,2} a_j^2
        + \ldots + s_{j,l,D_s} a_j^{D_s}.
\end{equation*}

The pairs $(g_j^{l-1}, u_j^{l-1})$ and $(g_j^l, u_j^l)$ denote
respectively the coordinates of the initial and final points of the piece $l$
of the spline. The points $g_j^l$ for $l=1$ to $k-1$ partition the criterion domain $[v_{1,j}, v_{2,j}]$ in subintervals. We set $v_{1,j} = g_j^0$ and  $v_{2,j} = g_j^k$. Hence the piece $s_{j,l}$ of the spline is defined on the interval $[g_j^{l-1}, g_j^l]$. The
spline $s_{j,l}$ takes the value $u_j^{l-1}$
(resp. $u_j^l$) on  $g_j^{l-1}$ (resp. $g_j^l$).
The continuity of the spline at the connection points is ensured by
imposing the two following constraints:
\begin{equation*}
\left\{
\begin{array}{rclr}
    s_{j,l}(g_j^{l-1}) & = & u_j^{l-1}  & l = \{1, \ldots, k\},\\
    s_{j,l}(g_j^l) & = & u_j^l      & l = \{1, \ldots, k\}.
\end{array}
\right.
\end{equation*}
Usually, the continuity of the marginals is ensured up to the second
derivative so that slope and concavity at the connection points remain
continuous.
To ensure the continuity of the first derivative, the following
constraints are added:
\begin{equation*}
    s'_{j,l}(g_j^l) = s'_{j,l+1}(g_j^l) \quad l = \{1, \ldots, k-1\}.
\end{equation*}
Similarly, the following constraints are added to ensure the continuity of
the second derivative:
\begin{equation*}
    s''_{j,l}(g_j^l) = s''_{j,l+1}(g_j^l) \quad l = \{1, \ldots, k-1\}.
\end{equation*}
Of course, it is possible to ensure the continuity of the second
derivative only if the marginal polynomials have a degree equal to or higher
than 3.
More generally, it is possible to ensure the continuity of the polynomials
up to the $i^{\text{th}}$ derivative only if the polynomials have a degree
equal to or higher than $i+1$.

As in \utap{}, the main difficulty in \utasp{} is to find polynomials which
ensure the monotonicity of the marginals.
To achieve this, we use the results set out in Section
\ref{subsection-semidefinite_programming}.
Recall that the non-negativity of a univariate polynomial is ensured if it
can be expressed as a sum of squares.
The monotonicity of the marginals is therefore ensured by imposing the
non-negativity of their derivatives on an interval.
Formally, for the piece $l$ of the spline associated to criterion $j$, it
reads:
\begin{equation*}
    s'_{j,l}(a_j) = s_{j,l,1} + 2 s_{j,l,2} a_j + \ldots
        + D_s s_{j,l,D_s} a_j^{D_s - 1} \geq 0.
\end{equation*}
We impose $s'_{j,l}(a_j)$ to be a sum of two \sos{} as specified in Theorem
\ref{theorem-sos_interval}.
Formally it reads:
\begin{equation*}
    s'_{j,l}(a_j) = (x - g_j^{l-1}) \cdot q_{j,l}(a_j)
            + (g_j^l - x) \cdot r_{j,l}(a_j),
\end{equation*}
with $q_{j,l}(a_j)$ and $r_{j,l}(a_j)$ two polynomials that can be
expressed as sums of squares.

Using semidefinite programming, we impose two square matrices $Q_{j,l}$ and
$R_{j,l}$ of size $d = \left \lceil \frac{D_s - 1}{2} \right \rceil + 1$
to be positive semidefinite.
Hence, $q_{j,l}(a_j) = \overline{a_j}^\transp Q_{j,l} \overline{a_j}$ and
$r_{j,l}(a_j) = \overline{a_j}^\transp R_{j,l} \overline{a_j}$, with
$\overline{a_j}^\transp = \begin{pmatrix} 1 & a_j & \ldots &
a_j^d\end{pmatrix}$, are two non-negative polynomials.

The value of the polynomial coefficients $s_{j,l,0}, \ldots, s_{j,l,D_s}$
are obtained by combining the off-diagonal terms of the matrices.

\subsubsection{Link between \utasp{}, \utap{} and \uta{}}

We note that \utasp{} is a generalization of \uta{}.
Indeed, \uta{} is a particular case of \utasp{} in which splines of
the first degree are used.

A similar link exists between \utasp{} and \utap{}.
Indeed, if \utasp{} is used to learn marginals composed of exactly one
piece then it is equivalent to the \utap{} formulation.

%% file: exampleRev.tex
% vim:spell spelllang=en

\section{Illustrative example}

In this section, we illustrate \utap{} and \utasp{} on an small instance of a ranking problem.
In the first subsection we briefly present the context of the problem.
Then we infer the parameters of \utap{} models and compare the marginals
obtained with \utap{} to the original ones.
Finally we perform the same experiment with \utasp{}.
To formulate and solve the SDP we used CVX, a Matlab software for
\textit{disciplined convex programming}\cite{grantboyd2014}.
The source code of \utap{} and \utasp{} is available at the
following address: \url{http://olivier.sobrie.be}.

\subsection{Context of the problem}

A family plans to spend a one week holiday in France.
They use a search engine which
returns a list of 1000 possible accommodations.
To avoid reviewing the whole list and save time, the family calls  a
\mcda{} analyst.
The first task of the analyst consists in determining which criteria matter to the family.
They identify the following three criteria:
\begin{itemize}
	\item Price: the price of the renting in euros which should be
	minimized;

	\item Distance: the distance from home in kilometers which should
	be minimized;

	\item Size: the size of the accommodation in square meters which
	should be maximized.
\end{itemize}
The family cannot evaluate the importance of the criteria and doesn't want to enter into a formal elicitation procedure. On the contrary, they are ready to make some overall statements that could be used by a model learning method.

Let us assume that the preferences of the family can be represented by an additive value function and that the marginals are
displayed in Figure \ref{fig-plot_app_real_utilities}.
These functions are polynomials of degree 2 ($u_1$ and $u_3$) and 15 ($u_2$).
\begin{figure}[!h]
	\centering
	\includegraphics[scale=1]{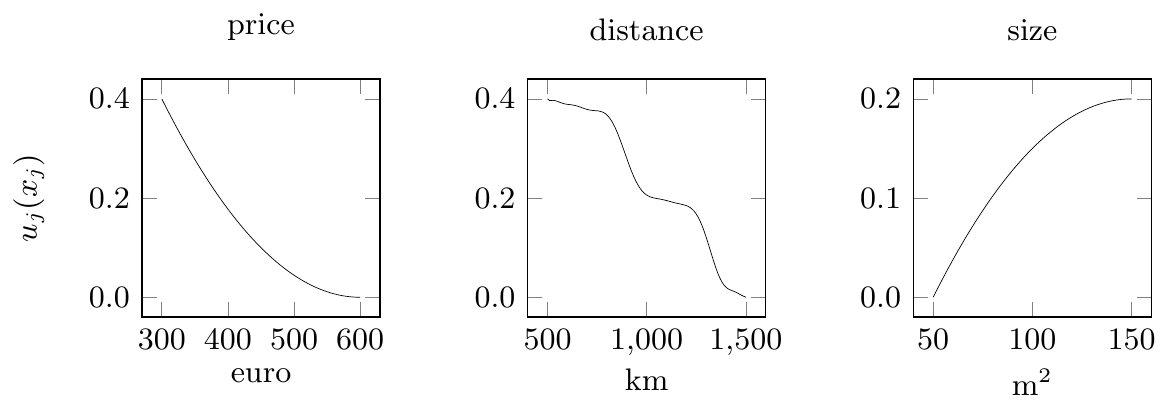}
	\caption{True marginal value functions modeling the family's preferences.}
	\label{fig-plot_app_real_utilities}
\end{figure}

\subsection{\utap{}}

In order to learn the marginals given in Figure
\ref{fig-plot_app_real_utilities}, the family ranks a subset of 50
alternatives chosen randomly in the list according to the
unveiled marginal functions displayed in Figure
\ref{fig-plot_app_real_utilities}.

The 49 informative pairwise comparisons are used to learn, using \utap{}, an
additive value function model with polynomials of degree one to ten.
The inferred value function yields a ranking of the 50 alternatives.
Hence, we can observe the similarity of the initial and inferred
rankings. The evolution of the Spearman distance and Kendall Tau of
these rankings is given in Figure \ref{fig-plot_app_utap_degree} .
We observe that increasing the degree of the polynomial increases
the accuracy of the model. Indeed, the values of the Spearman
distance and Kendall Tau grow as a function of the degree of the
marginals.
%We note that the Spearman distance is worse for $D=5$
%than for $D=4$. However it is not the case of the Kendall Tau, it is
%better with $D=5$ than with $D=4$. We explain this phenomenon by the
%objective function of the \sdp{} program. Indeed its objective does
%minimize a slack that tends to be equal to 0 when the ranking is
%restored. Therefore it can be subject to compensatory effects
%between pairwise comparisons.

In a second step, the analyst asks to the family to include 50 other
alternatives in the ranking.
The analyst provides a set of 99 pairwise comparisons to \utap{}.
As in the first step, polynomials of degree one to ten are learned.
We observe in Figure \ref{fig-plot_app_utap_degree} that the accuracy of
the model is improved with more pairwise comparisons when the marginals
have a small degree (smaller than 8).
With more examples we see that the Spearman distance and Kendall Tau are
slightly better when marginals degree is small and slightly worse
when marginals degree is superior to 9.

\begin{figure}
    \centering
    \includegraphics{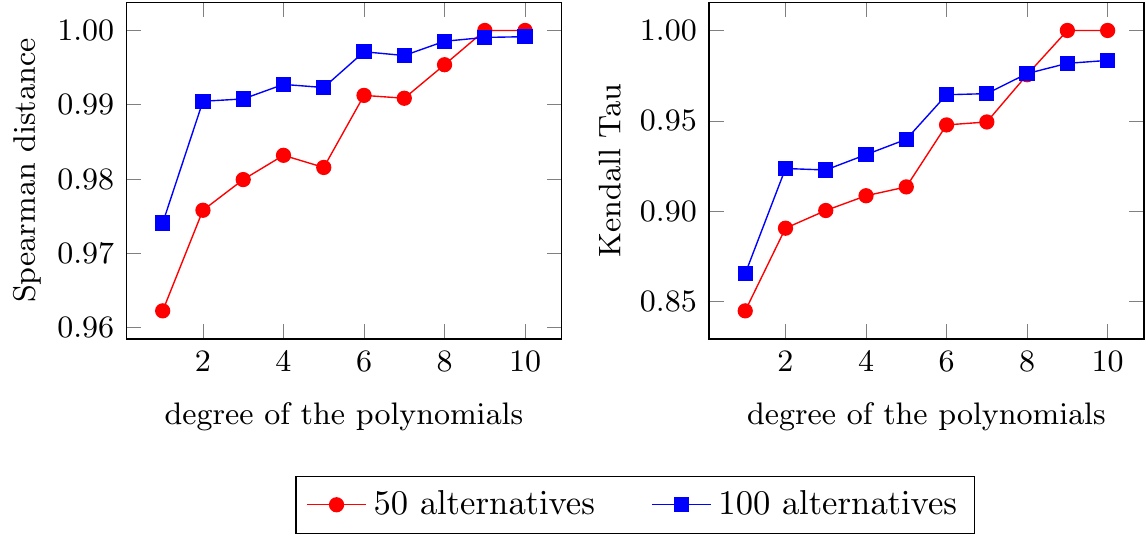}
    \caption{Evolution of the Spearman distance and Kendall Tau of
    the learning set as a function of the degree of the marginal
    polynomials for learning sets composed of 50 and 100 examples.}
    \label{fig-plot_app_utap_degree}
\end{figure}

For illustrative purpose we show in Figure
\ref{fig-plot_app_utasp_1piece} the marginals learned on basis of 100
examples with polynomials of degree 2, 6 and 10.
We see that the marginals $u_1$ and $u_3$ are well approximated with
polynomials of degree 2 to 10.
The major difference is observed for $u_2$.
Using a polynomial of degree 2 approximates roughly the curve.
The two steps of $u_2$ cannot be better approximated by a polynomial of the
second degree since there is no inflexion point with such a polynomial.
The real marginal has at least two inflexion where the steps are located.
With a polynomial of degree 6 we see that the approximation of this curve
is improved but it does not perfectly fit the real marginal.
Indeed the slope is less steep between the inflexion points.
With a polynomial of degree 10 the learned marginal almost perfectly fit
the real marginal.
The inflexion of the curve happens at the same places and the slopes are
similar.

\begin{figure}
    \centering
    \includegraphics{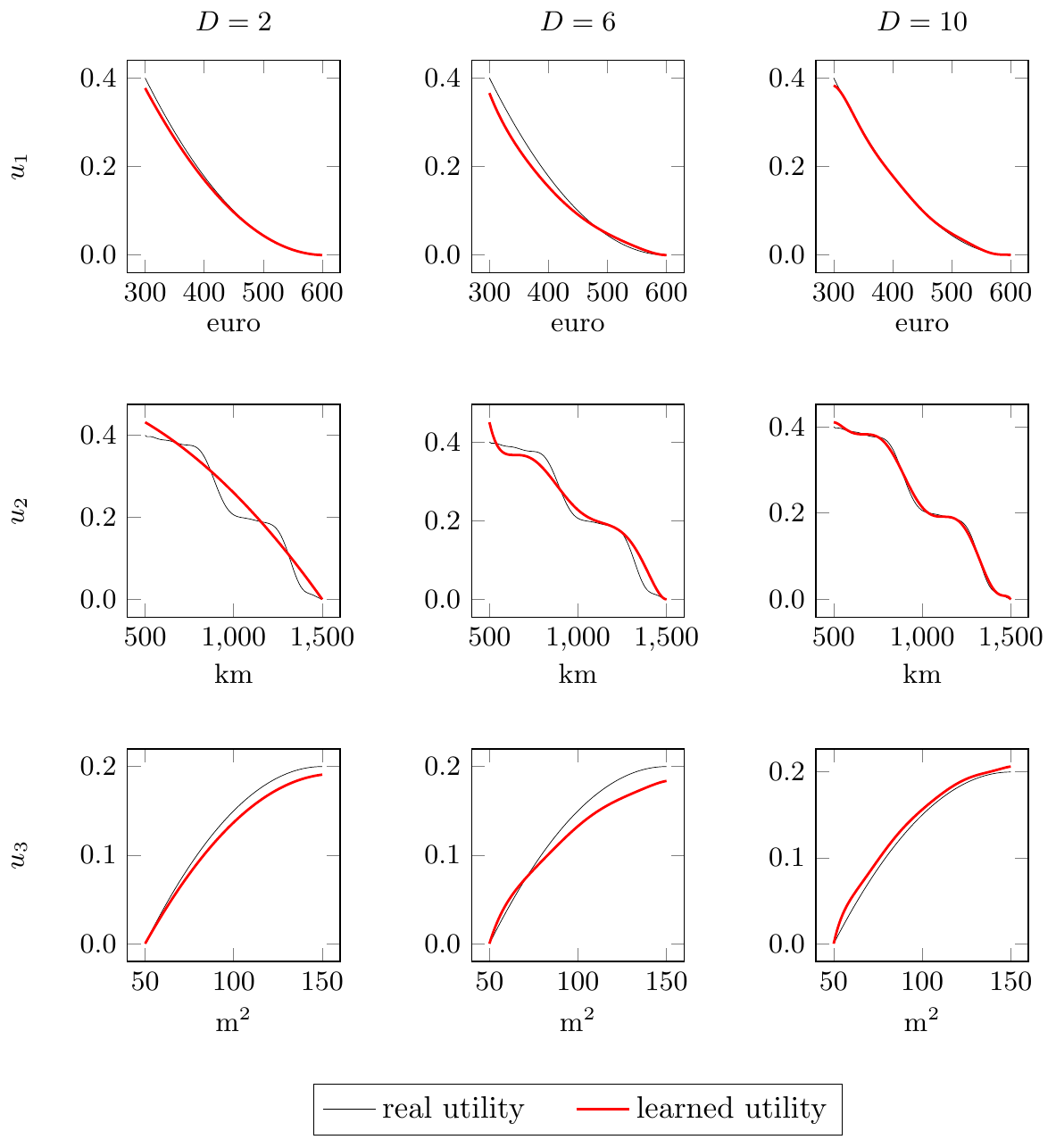}
    \caption{Value functions learned by \utap{} on basis of a
    learning set composed of 100 examples with polynomials of degree
    $D = 2, 6$ and 10.}
    \label{fig-plot_app_utasp_1piece}
\end{figure}

\subsection{\utasp{}}

As for \utap{}, we perform some experimentations with \utasp{} on the application described above. We vary the number of pieces and the polynomial degrees of
\utasp{} and observe the variation in accuracy. We also study the impact of the continuity degree on the splines.

Figure \ref{fig-plot_app_utasp_2to5pieces_degree_sdkt} shows the evolution of the average Spearman distance and Kendall Tau on the learning set. We note that
increasing the number of pieces usually has a positive influence on the way \utasp{} succeeds in restoring the original ranking. \utasp{} is able to
restore the original ranking with smaller polynomial degrees when the number of pieces increases. However it is not always the case. For instance, when using polynomials of degree 1, a \uta{} model composed of 4 pieces performs better than one using 5 pieces. With polynomials of degree greater than 1, \utasp{} always
performs better when the number of pieces is larger.

For illustrative purpose, we show in Figure \ref{fig-plot_app_utasp_5pieces} the marginals obtained with splines of degree $D= 1$ to 3. The continuity of the
splines at the breakpoints ($D_c$) is enforced up to $D-1$. With polynomials of degree 3, we observe that the learned marginals tightly fit the real
marginals.

\begin{figure}
    \centering
    \includegraphics{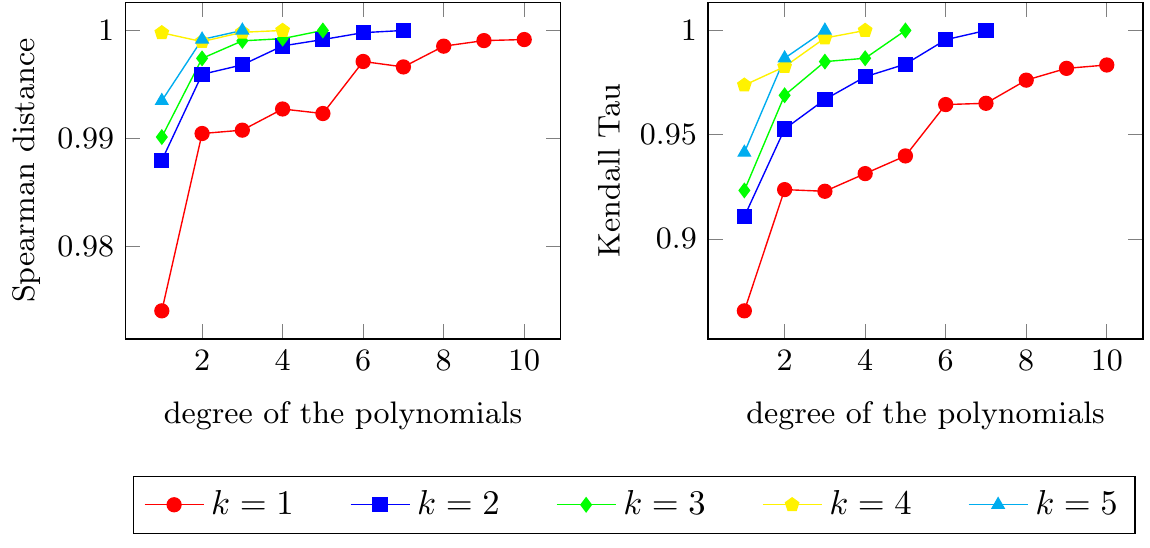}
    \caption{Evolution of the Spearman distance and Kendall Tau of
    the learning set as a function of the degree of the
    marginal polynomials for learning sets composed of 100 examples
    with 1 to 5 polynomials per marginal.}
    \label{fig-plot_app_utasp_2to5pieces_degree_sdkt}
\end{figure}

\begin{figure}
    \centering
    \includegraphics{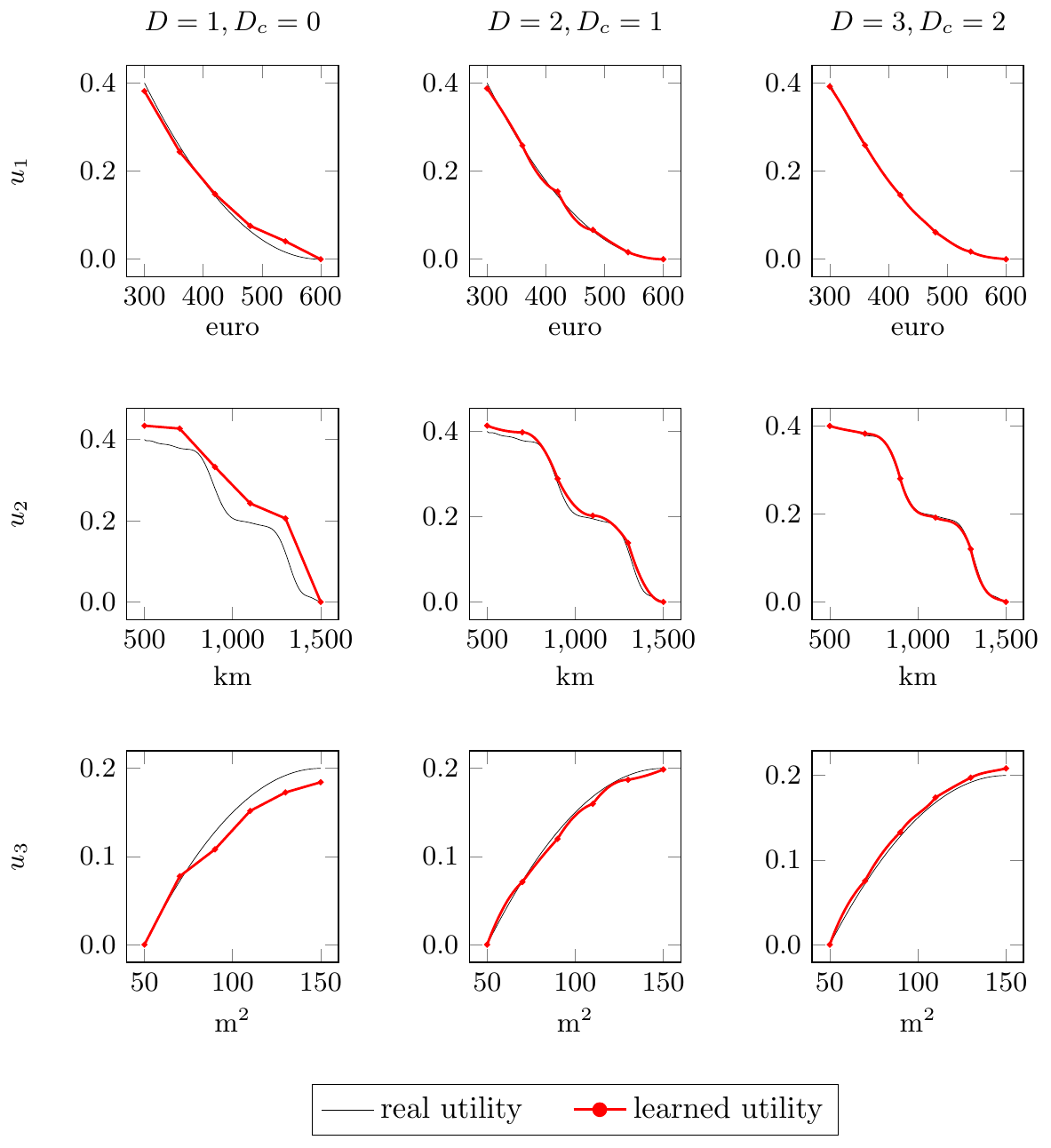}
    \caption{Value functions learned by \utasp{} on basis of a
    learning set composed of 100 examples with polynomials of degree
    $D = 1$ to 3 and marginals composed of 5 polynomials ($k=5$). The
    continuity of the spline ($D_c$) is enforced up to $D-1$.}
    \label{fig-plot_app_utasp_5pieces}
\end{figure}

%% file: experimentsRev.tex
% vim:spell spelllang=en

\section{Experiments}
\label{section-experiments}

So as to understand the behavior of \utap{} and \utasp{}, we
performed experiments on artificial datasets. These experiments aim
at studying the ability of the methods to retrieve a ranking from a
set of pairwise comparisons and the computing time. In the
experiments, we vary different parameters of \utap{} and \utasp{}:
degree of the polynomials ($D$), number of pieces ($k$), the
continuity at breakpoints ($D_c$) and the number of alternatives in
the learning set ($m^*$). As in the previous Section, we formulate
and solve the SDP we used CVX, a Matlab software for
\textit{disciplined convex programming}\cite{grantboyd2014}.

\subsection{Experimental setup}

Our experimental strategy is the following. We start from an
hypothetical additive value model denoted $M$, and generate a set of
alternatives (called learning set). Then we simulate the behavior of
a DM ranking these alternatives, while having the model $M$ in mind.
Hence, we constitute a ranking on the learning set.

We compute an additive value model using \utap{} and \utasp{}
compatible with the ranking of the learning set. We then compare the
inferred models to the model $M$. To do so, we randomly generate
another set of alternatives (test set), and we compute the ranking
of this test set obtained by the model $M$ and by the inferred
model. We then compute the Spearman distance \cite{spearman1904} and
the Kendall Tau \cite{kendall1938} to evaluate how close the inferred
rankings are to the original one.

%In order to test \utap{} and \utasp{}, we generated artificial
%learning and test sets. These learning and test set consist in a set
%of alternatives evaluated on $n$ criteria ranked from the worst to
%the best. Eight initial models have been used in order to generate
%the learning and test sets.

We considered 8 different models $M$, chosen to represent a wide
variety of value functions (structure and forms of the marginals).
Four of these models are composed of 3 criteria (Figure
\ref{fig-plot_expe_models_marginals}), while the four others are
composed of 5 criteria (Figure
\ref{fig-plot_expe_models_marginals_5criteria}). As shown in Figure
\ref{fig-plot_expe_models_marginals} and
\ref{fig-plot_expe_models_marginals_5criteria}, the marginals are
 of different type: piecewise linear functions, sigmoids,
exponentials, and polynomials of degree 2, 3 and 15.

\begin{figure}[p]
    \centering
    \includegraphics{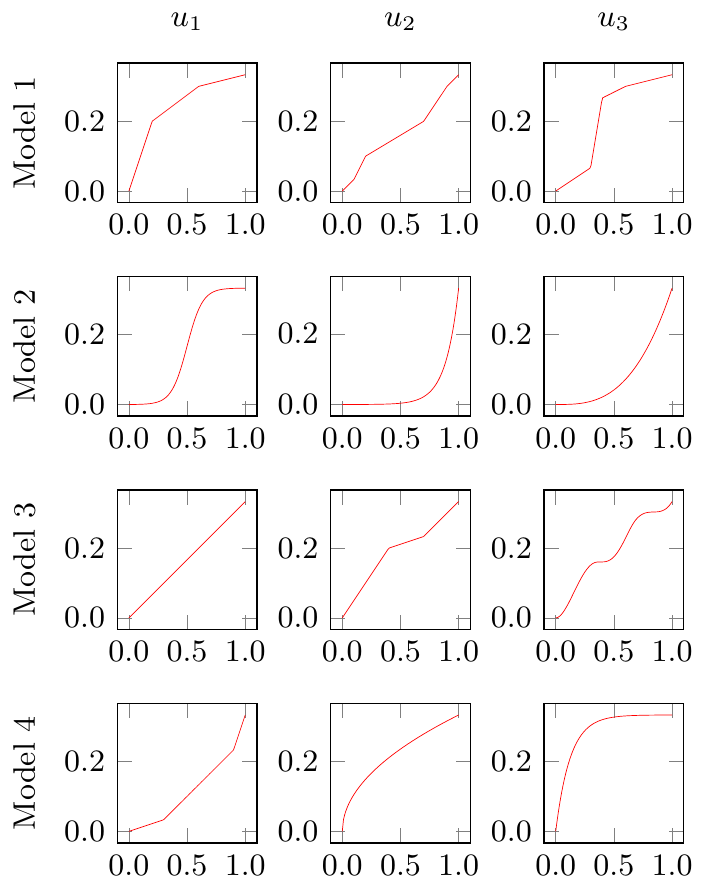}
    \caption{Four additive value function models composed of 3
    criteria.\label{fig-plot_expe_models_marginals}}
\end{figure}

\begin{figure}[p]
    \centering
    \includegraphics{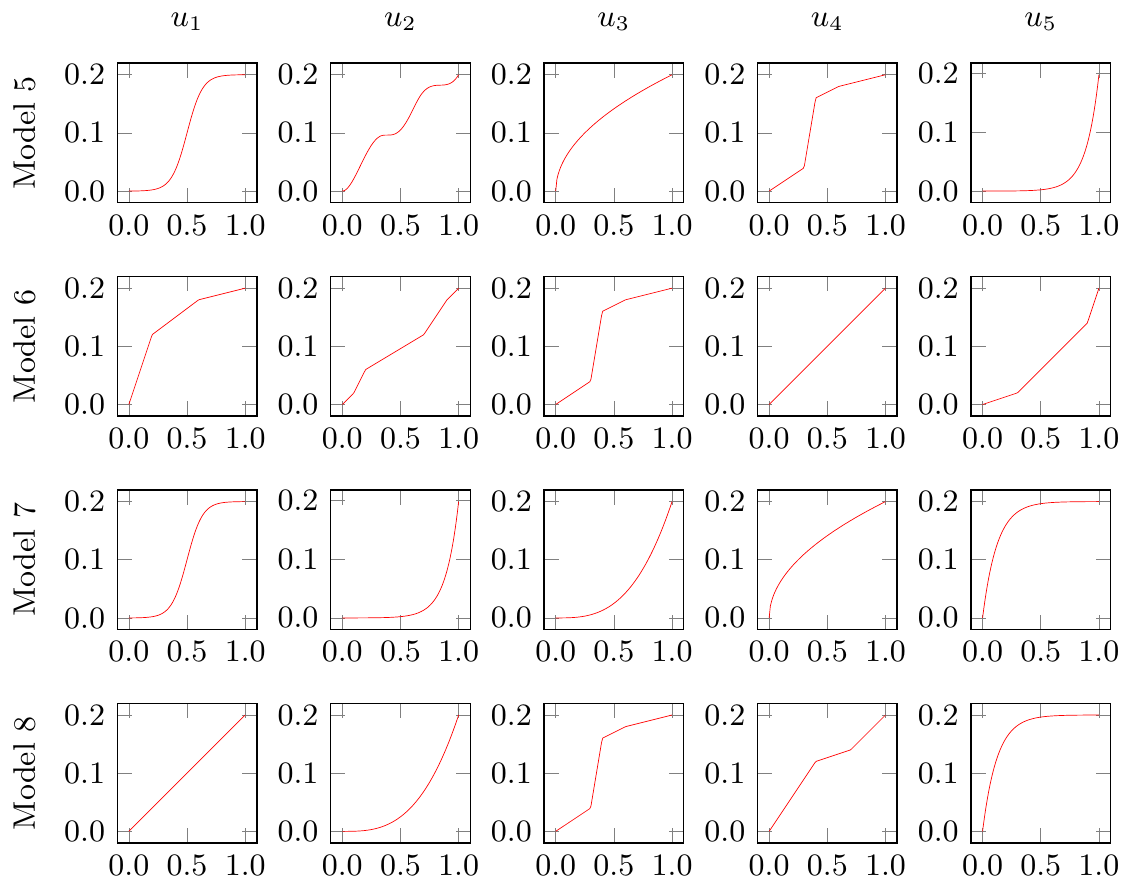}
    \caption{Four additive value function models composed of 5
    criteria.\label{fig-plot_expe_models_marginals_5criteria}}
\end{figure}

For a given model $M$ and a seed $s$, the experimental procedure is
the following:
\begin{enumerate}
    \item The random generator is initialized with the seed $s$.

    \item A set of $m^*$ performances vectors (alternatives) is generated.
    It constitutes the learning set $A^*$.
    Each component $a^*_j$ of a performances vector $a^*=(a^*_1, a^*_2, ..., a^*_n) \in A^*$
    is generated by drawing $n$ a random number uniformly in $[0, 1]$.

    \item The score $U(a^*)$ is computed for each vector of performances
    $a^* \in A^*$ using the value model $M$. A pre-order on these
    alternatives is derived from their scores.
%    The higher the score of an alternative, the higher its position in
%    the ranking.
%    If two alternative have the same score then they are classified at
%    the same position in the ranking.
    Given a ranking $\pi^*$ of the alternatives in $A^*$, we denote by
    $\pi^*_i$ the alternative ranked at the $i^{\text{th}}$ position.
    We have $\pi^*_1 \succcurlyeq \pi^*_2 \succcurlyeq \ldots
    \succcurlyeq \pi^*_{m^*}$.

    \item A list of $m^*-1$ pairwise comparisons is induced from the
    complete ranking $\pi^*$.
    It is done by comparing each pair of consecutive alternatives in
    the ranking.
    In a ranking $\pi^*$, it consists in comparing $\pi^*_i$ to $\pi^*_{i+1}$, either by an indifference ($\pi^*_i
    \sim \pi^*_{i+1}$) or a preference ($\pi^*_i \succ
    \pi^*_{i+1}$).
    We denote by $\mathcal{P}^*$ the set containing the pairs of
    alternatives $(a,b)$ such that $a \succ b$, $\mathcal{I}^*$
    denotes the set containing the pairs $(a,b)$ such that $a \sim b$.

    \item The sets $A^*$, $\mathcal{P}^*$ and $\mathcal{I}^*$ are given as input to
    \utasp{}/\utap{}. The algorithm learns an additive utility model $M'$ in which the
    marginals are composed $k$ polynomials of degree $D$.
    The breakpoints of the polynomials are equally spaced on the
    criterion domain.
    The continuity is guaranteed up to the $D_c^{\text{th}}$
    derivative at the breakpoints.

%    \item Alternatives in $A^*$ are ranked according to the model
%    $M'$.
%    We denote by $\hat{\pi}^{*}$ the ranking obtained with $M'$.
%    To assess the quality of the model, the rankings $\pi^*$ and
%    $\hat{\pi}^{*}$ are compared by computing the Spearman distance
%    \cite{spearman1904} and the Kendall Tau \cite{kendall1938}.
%    In the sequel, we denote by $\pi^*(a)$ the rank of the alternative
%    $a$.
%    We define by $D_s(\pi^*, \hat{\pi}^*)$ the sum of squared distance
%    between rankings $\pi^*$ and $\hat{\pi}^*$ such that
%    $D_s(\pi^*, \hat{\pi}^*) = \sum_{a \in A} (\pi^*(a) -
%    \hat{\pi}^*(a))^2$.
%    The Spearman distance, denoted by
%    $\mathit{SD(\pi^*,\hat{\pi}^*)}$, is defined as:
%    \begin{equation*}
%        \mathit{SD}(\pi^*,\hat{\pi}^*)
%        = 1 - \frac{6 \cdot D_s(\pi^*,\hat{\pi}^*)}{|A|(|A|^2 - 1)}.
%    \end{equation*}
%    Let $D_\tau(\pi^*, \hat{\pi}^*)$ denotes the sum of alternative pairs that
%    are inverted in the ranking obtained with the learned model.
%    Formally, it reads $D_\tau(\pi^*, \hat{\pi}^*) = \left| {a,a'} :
%    \pi^*(a) < \pi^*(a') \land \hat{\pi}^*(a) > \hat{\pi}^*(a) \right|$.
%    The Kendall Tau, denoted by $\mathit{KT}(\pi^*,\hat{\pi}^*)$, is
%    defined as:
%    \begin{equation*}
%        \mathit{KT}(\pi^*,\hat{\pi}^*)
%        = 1 - \frac{4 \cdot D_\tau(\pi^*,\hat{\pi}^*)}{|A|(|A| - 1)}.
%    \end{equation*}

    \item A test set of $m$ alternatives $A$ is generated
    similarly as for the learning set.
    The alternatives in $A$ are ranked with models $M$ and $M'$.
    The obtained ranking $\pi$ and $\hat{\pi}$ are then compared by
    computing the Spearman distance $\mathit{SD}(\pi,\hat{\pi})$ (see \cite{spearman1904}) and
    the Kendall Tau $\mathit{KT}(\pi,\hat{\pi})$ (see \cite{kendall1938}).
\end{enumerate}

\subsection{Model retrieval}

We tested \utap{} and \utasp{} with the models shown in Figures
\ref{fig-plot_expe_models_marginals} and
\ref{fig-plot_expe_models_marginals_5criteria}. Results provided in
this Section are mean values over the 8 different models tested. We
varied the degree of the polynomials ($D$), the number of pieces
($k$), the continuity at the breakpoints ($D_c$). We varied the size
of the learning set ($m^*$) between 10 and 100 alternatives. The
test set was composed of 1000 alternatives. For each setting, we ran
the test procedure described above with 10 random seeds.

This experiment shows how the number of comparisons impacts the
ability to elicit the parameters of a model $M$ composed of $n$
criteria. The experiment also shows the impact of the number of
pieces per marginal and of the degree of the polynomial.

\subsubsection{\utap{}}

The first test consists in testing \utap{} with only one piece per
marginal ($k = 1$). We show in Figure
\ref{fig-plot_expe_utasp_1piece_test_avg} the average Spearman
distance and Kendall Tau of the test set of the models composed of 3
criteria when the degree of the learned marginals ($D$) vary from 1
(which corresponds to a weighted sum) to 4. The values of the
Spearman distance and Kendall Tau increase as a function of the
number of alternatives in the learning set. For the same number of
examples in the learning set, the quality of the ranking is improved
as the degree of the polynomial increases. We observe the same
behavior with models composed of 5 criteria (Figure
\ref{fig-plot_expe_utasp_5criteria_1piece_test_avg}). Detailed
results per model are available in \ref{section-detailed_results}.

\begin{figure}
    \centering
    \includegraphics{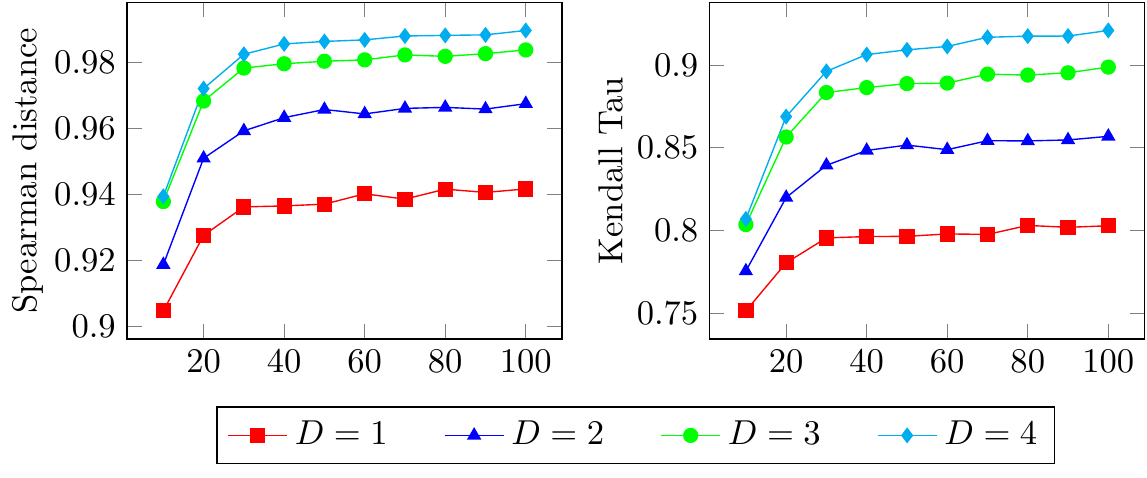}
    \caption{Average Spearman distance and Kendall Tau of the test set
    with the models composed of 3 criteria learned by \utap{} when the
    degree of the marginals vary between 1 and 4.}
    \label{fig-plot_expe_utasp_1piece_test_avg}
\end{figure}

\begin{figure}
    \centering
    \includegraphics{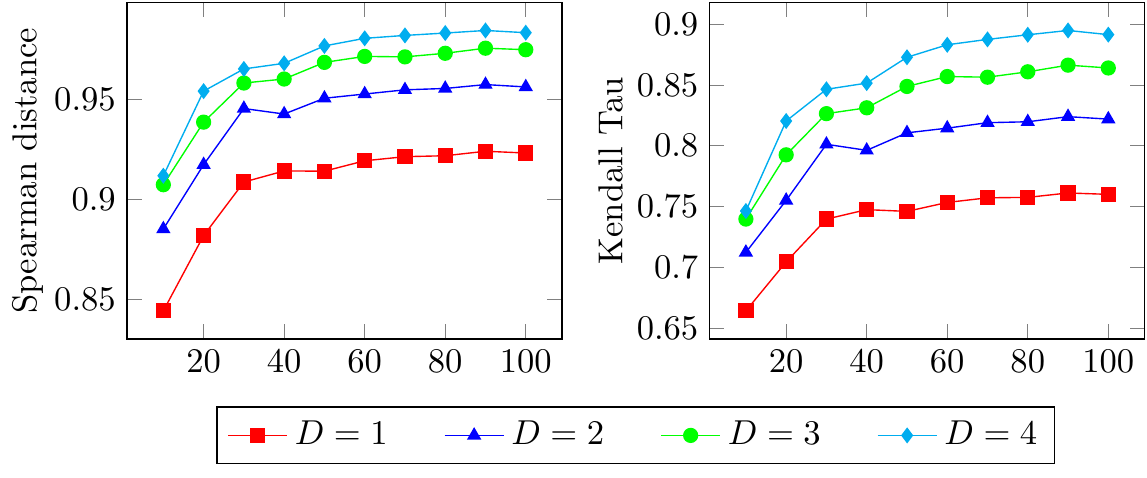}
    \caption{Average Spearman distance and Kendall Tau of the test set
    with the models composed of 5 criteria learned by \utap{} when the
    degree of the marginals vary between 1 and 4.}
    \label{fig-plot_expe_utasp_5criteria_1piece_test_avg}
\end{figure}

\subsubsection{\utasp{}}

In the second test, we varied the number of pieces per marginals
($k$) from 1 to 5 and used polynomials of degree 3. The continuity
at the breakpoints is ensured up to the second derivative. Figure
\ref{fig-plot_expe_utasp_npieces_test_avg} shows the average
Spearman distance and Kendall Tau of the test set for the models
composed of 3 criteria. We observe that increasing the number of
pieces helps to increase the accuracy of the model. With models
composed of 5 criteria (see
Figure~\ref{fig-plot_expe_utasp_npieces_5criteria_test_avg}), we
observe the same behavior. It depicts a general trend for the model
presented in Figure \ref{fig-plot_expe_models_marginals} and
\ref{fig-plot_expe_models_marginals_5criteria}. Nevertheless one has
to be cautious to overfitting effects when the number of pieces
increases and to the position of the breakpoints. Indeed increasing
the number of pieces increases the number of parameters of the model
and its flexibility which may lead to overfitting. In
\ref{section-detailed_results} we present the detailed results for
each model of Figure \ref{fig-plot_expe_models_marginals} and
\ref{fig-plot_expe_models_marginals_5criteria}.

\begin{figure}[!h]
    \centering
    \includegraphics{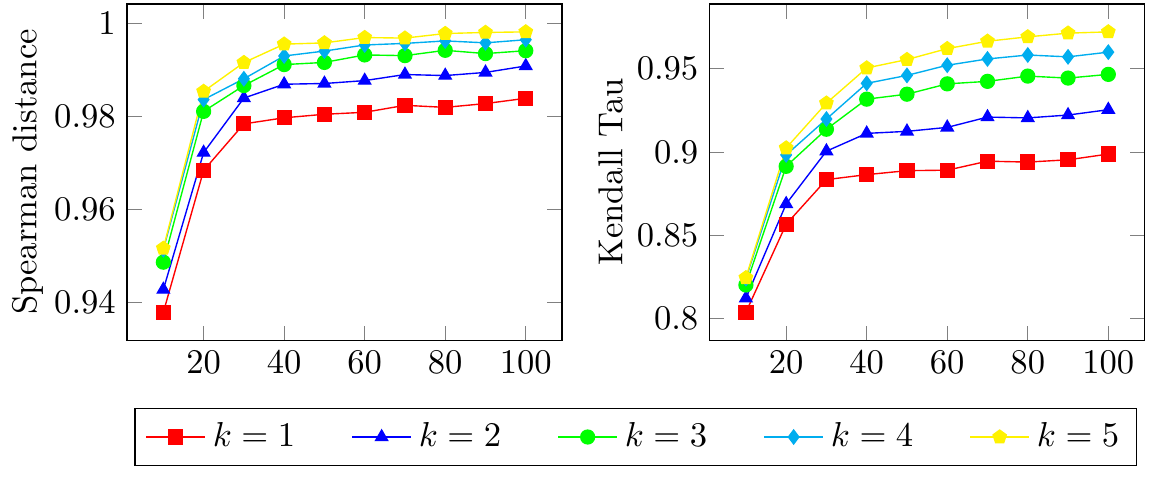}
    \caption{Average Spearman distance and Kendall Tau of the test set
    with the models composed of 3 criteria learned by \utasp{} with
    marginals composed of polynomials of the third degree. The
    continuity at the breakpoints is ensured up to the second
    derivative.}
    \label{fig-plot_expe_utasp_npieces_test_avg}
\end{figure}

\begin{figure}[!h]
    \centering
    \includegraphics{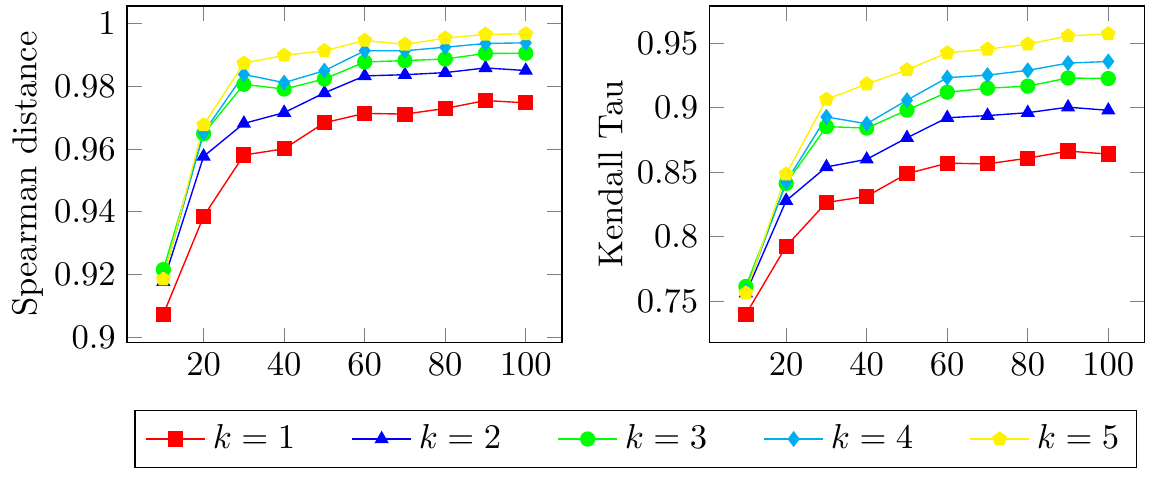}
    \caption{Average Spearman distance and Kendall Tau of the test set
    with the models composed of 5 criteria learned by \utasp{} with
    marginals composed of polynomials of the third degree. The
    continuity at the breakpoints is ensured up to the second
    derivative.}
    \label{fig-plot_expe_utasp_npieces_5criteria_test_avg}
\end{figure}

\subsection{Computing time}

The computing time highly depends on the number of constraints and
variables that are involved. The number of constraints and variables
are expressed by the following equations:
\begin{align*}
    \text{\#constraints}
        & = m + n + 2 n k + n k D
            + (1 + D_c) n (k - 1),\\
    \text{\#variables}
        & = n k (D + 1)
        + 2 n k \left\lceil \frac{D}{2} \right\rceil^2
        + 2m.
\end{align*}
We give in Table \ref{table-nconstraintsvars} the number of constraints
and variables for different problem sizes.

\begin{table}
\centering
\begin{tabular}{rrrrrrrr}
    \toprule
    $m$ & $n$ & $k$ & $D$ & $D_c$ & \#const. & \#var.
        & computing time (sec.)\\
    \midrule
    10 & 3 & 1 & 1 & 0 & 22 & 32 & $0.48 \pm 0.15$\\
    10 & 3 & 5 & 1 & 0 & 70 & 80 & $1.02 \pm 0.34$\\
    10 & 3 & 1 & 4 & 0 & 31 & 59 & $0.86 \pm 0.19$\\
    10 & 3 & 5 & 4 & 2 & 139 & 215 & $1.96 \pm 0.29$\\
    10 & 5 & 5 & 4 & 2 & 225 & 345 & $2.99 \pm 0.36$\\
    100 & 3 & 1 & 1 & 0 & 112 & 212 & $1.96 \pm 0.14$\\
    100 & 3 & 5 & 1 & 0 & 160 & 260 & $2.58 \pm 0.14$\\
    100 & 3 & 1 & 4 & 0 & 121 & 239 & $2.96 \pm 0.14$\\
    100 & 3 & 5 & 4 & 2 & 229 & 395 & $3.92 \pm 0.20$\\
    100 & 5 & 5 & 4 & 2 & 315 & 525 & $5.90 \pm 0.35$\\
    \bottomrule
\end{tabular}
\caption{Number of constraints and variables for different problem sizes
and average computing time and standard deviation.}
\label{table-nconstraintsvars}
\end{table}

We observe that the computing time evolves linearly with the number
of examples that are given as input to the algorithm. For the
inference of a \utap{} model, the higher the degree of the
polynomials, the higher computing time; however the difference is
not substantial. Compared to an \uta{} model, learning a \utap{}
model using polynomials of the \nth{4} degree increases the
computing time of a few dozen of milliseconds. The behavior is
similar when passing from one to several pieces per marginal. When
the number of criteria increases, we observe that the computing time
increases too.

Lastly, it should be highlighted that computing times for all
instances solved in this Section are reasonably short (less than 6
sec.), and compatible with an iterative and interactive use with a
DM.

%% file: concluRev.tex
% vim:spell spelllang=en

\section{Conclusion}

In this paper, we propose a new method to learn an additive value
function model from a set of statements provided by the DM. Learning
piecewise linear value functions from preference statements is
standard in the literature (UTA methods, e.g.
\cite{jaquetlsiskos1982}, \cite{jaquetlsiskos2001}). Instead of
piecewise linear marginals, we generalize this standard
representation by considering more general forms for marginals.
\utap{} considers marginal value functions which are monotone
polynomials, while in \utasp{} marginals are composed of several
pieces of monotone polynomials.
\utasp{} generalizes the preference representation used in the standard UTA
methods, while \utap{} is a particular \utasp{} model where a single polynomial is used to represent each marginal.

The inference of such an additive value function with polynomial
marginals is performed using a semidefinite programming formulation.
From a computational point of view, the resolution of instances
corresponding to real datasets is limited to several seconds, and
thus compatible with an interactive use with DMs.

We provide an illustrative example showing that the inference
program is able to restore value functions that are ``close'' to the
original ones. A specific feature of the methods is that the
inferred value function is composed of  ``smooth'' marginals which
avoids brutal changes in the slopes of these marginals, thus
improving interpretability.

The computational experiments show the ability of the methods to
better match the preference statements as the degree of the
polynomials involved in the marginals increases.

It should be noted that the methods proposed in this paper, applies
to ranking problems but can be directly extended to sorting
problems, hence defining \utadisp{} and \utadissp{} (see \cite{sobrie2016}).

An innovative aspect of this work is related to the new optimization
technique allowing to deal with polynomial and piecewise polynomial
marginals instead of piecewise linear marginals. The semidefinite
programming approach used in this paper for UTA opens new
perspectives for eliciting other preference models based on additive
or partly additive value structures, such as additive differences
models (MACBETH \cite{bana1994macbeth,bana2005}) and GAI networks
\citep{GonzalesEtAl2011}.

Similarly as for UTA models (cf.\@ the discussion in the Introduction),
the solution of our new models might not be unique. It would be
interesting to try to characterize these situations and pick a solution
that is most suited for the DM. Note that, for this work, we used
interior-point methods to solve the semidefinite programs. These methods
return the so-called analytic center of the set of optimal solutions, that
is, it returns a solution `in the middle' of the set of optimal solutions,
similarly as UTA-STAR and ACUTA would do for UTA models.

An interesting line for further research concerns the experimental
comparison of \utap{} and \utasp{} with classical UTA methods, in
particular in what concerns the size of the set of reference
alternatives required to adequately elicit the preference model.
Another area of interest for research concerns the extension of
the present methods to the paradigm of Robust Ordinal Regression.
It would be interesting to investigate how to identify the most
``simple''\footnote{Simplicity is hard to define precisely, but is
related to having polynomials with the smallest possible degrees
and no more changes of concavity than required} value function
compatible with the preference information ; this could be done by
introducing a regularization term in the objective function.
Lastly, when the set of preference statements is not representable
by a given preference model (\utap{}, with a given degree of
polynomials, \utasp{}, with a given degree and given number of
pieces), the issue of solving inconsistencies in the spirit of
\cite{mous5ejor03} is worth further investigation.

%% file: appendixRev.tex
% vim:spell spelllang=en

\section{Example of a semi-definite program}
\label{appendix-example_sdp}

We consider a ranking problem involving 2 criteria $x$ and $y$ and three
alternatives, $a^1$, $a^2$ and $a^3$.
The performances of these alternatives are given in Table
\ref{table-example_cvx_program_pt}.
The criterion values vary between 0 and 10.

\begin{table}[h]
\centering
\begin{tabular}{lrr}
	\toprule
		& $x$ & $y$\\
	\midrule
	$a^1$	& 10	& 7\\
	$a^2$	& 6	& 8\\
	$a^3$	& 7	& 5\\
	\bottomrule
\end{tabular}
\caption{Performances of alternative $a^1$, $a^2$ and $a^3$ on criteria
$x$ and $y$.}
\label{table-example_cvx_program_pt}
\end{table}

A decision maker states that the following ranking holds: $a^1 \succ a^2
\succ a^3$.
We use the objective and the set of constraints given in Equation
\eqref{eq-utadis_constraints} in order to
find a model restoring this ranking.
We use semi-definite programming to learn polynomial marginal utility
functions.
We denote by $u^*_1$ and $u^*_2$ the polynomial functions associated
respectively to criteria 1 and 2.
The degree of the polynomials of the marginal utility functions is fixed
to 3.

To ensure the monotonicity of functions $u^*_1$ and $u^*_2$, we impose
the non-negativity of their derivative.
Formally, we define $u^*_1$ and $u^*_2$ as follows:
\begin{align*}
	u^*_1(x) & = p_{x,0} + p_{x,1} \cdot x + p_{x,2} \cdot x^2
		+ p_{x,3} \cdot x^3,\\
	u^*_2(y) & = p_{y,0} + p_{y,1} \cdot y + p_{y,2} \cdot y^2
		+ p_{y,3} \cdot y^3.
\end{align*}
The derivative of $u^*_1(x)$ and $u^*_2(y)$ are equal to:
\begin{equation*}
	\frac{d u^*_1}{dx}
		=  p_{x,1} + 2p_{x,2} \cdot x +  3p_{x,3} \cdot x^2\\
	\text{\qquad and \qquad}\\
	\frac{d u^*_2}{d y}
		=  p_{y,1} + 2p_{y,2} \cdot y + 3p_{y,3} \cdot y^2.
\end{equation*}
The monotonicity of a polynomial marginal is ensured if its derivative is a
sum of square.
Formally, it reads:
\begin{align*}
	\frac{d u^*_1}{dx}
	& =
	\overline{x}^\transp Q \overline{x}\\
	& =
	\begin{pmatrix}
	1\\
	x
	\end{pmatrix}^\transp
	\begin{pmatrix}
	q_{0,0} & q_{0,1}\\
	q_{1,0} & q_{1,1}
	\end{pmatrix}
	\begin{pmatrix}
	1\\
	x
	\end{pmatrix}\\
	& = q_{0,0} + \left(q_{0,1} + q_{0,1}\right) x
		+ q_{1,1} x^2,\\
	\frac{d u^*_2}{d y}
	& =
	\overline{y}^\transp R \overline{y}\\
	& = r_{0,0} + \left(r_{0,1}
		+ r_{1,0}\right) y + r_{1,1} y^2.
\end{align*}
To ensure the non-negativity of the derivative, we impose the matrices
$Q$ and $R$ to be semi-definite positive in conjunction with
these constraints:
\begin{equation*}
\begin{cases}
	p_{x,1} & = q_{0,0},\\
	2 p_{x,2} & = q_{0,1} + q_{1,0},\\
	3 p_{x,3} & = q_{1,1},
\end{cases}\\
{\text{\qquad and \qquad}}\\
\begin{cases}
	p_{y,1} & = r_{0,0},\\
	2 p_{y,2} & = r_{0,1} + r_{1,0},\\
	3 p_{y,3} & = r_{1,1}.
\end{cases}
\end{equation*}
The utility values of $a^1$, $a^2$ and $a^3$ read:
\begin{align*}
	U(a^1) & = p_{x,0} + 10 p_{x,1} + 100 p_{x,2} + 1000 p_{x,3}
		+ p_{y,0} + 7 p_{y,1} + 49 p_{y,2} + 343 p_{y,3},\\
	U(a^2) & = p_{x,0} + 6 p_{x,1} + 36 p_{x,2} + 324 p_{x,3}
		+ p_{y,0} + 8 p_{y,1} + 64 p_{y,2} + 512 p_{y,3},\\
	U(a^3) & = p_{x,0} + 7 p_{x,1} + 49 p_{x,2} + 343 p_{x,3}
		+ p_{y,0} + 5 p_{y,1} + 25 p_{y,2} + 125 p_{y,3}.
\end{align*}
To find a model reflecting the ranking given as input, i.e. $a^1 \succ a^2
\succ a^3$, we have to fulfil two conditions: $a^1 \succ a^2$ and $a^2
\succ a^3$.
It is done by adding the following constraints:
\begin{equation*}
\left\{
\begin{array}{rcl}
	U(a^1) - U(a^2) + \sigma^+(a^1) - \sigma^-(a^1)
		- \sigma^+(a^2) + \sigma^-(a^2) & > & 0,\\
	U(a^2) - U(a^3) + \sigma^+(a^2) - \sigma^-(a^2)
		- \sigma^+(a^1) + \sigma^-(a^1) & > & 0.
\end{array}
\right.
\end{equation*}
After substituting $U(a^1)$, $U(a^2)$ and $U(a^3)$ by their value we
obtain the two following constraints:
\begin{equation*}
\left\{
\begin{array}{rcl}
	4 p_{x,1} + 64 p_{x,2} - p_{y,1} - 15 p_{y,2}
		+ \sigma^+(a^1) - \sigma^-(a^1)\\
	\multicolumn{1}{r}{- \sigma^+(a^2) + \sigma^-(a^2)} & > & 0,\\
	- p_{x,1} - 13 p_{x,2} + 3 p_{y,1} + 39 p_{y,2}
		+ \sigma^+(a^2) - \sigma^-(a^2)\\
	\multicolumn{1}{r}{- \sigma^+(a^3) + \sigma^-(a^3)} & > & 0.
\end{array}
\right.
\end{equation*}
Given that criteria domains are comprised between 0 and 10, the following
constraints hold:
\begin{equation*}
\left\{
\begin{array}{rcl}
	p_{x,0} & = & 0,\\
	p_{y,0} & = & 0,\\
	10 p_{x,1} + 100 p_{x,2} + 1000 p_{x,3}
		+ 10 p_{y,1} + 100 p_{y,2} + 1000 p_{y,3} & = & 1.
\end{array}
\right.
\end{equation*}
Finally, by assembling the objective function and the constraints, we
obtain the following semi-definite program:
\begin{equation*}
	\min{\sigma^+(a^1) + \sigma^-(a^1) + \sigma^+(a^2)
		+ \sigma^-(a^2) + \sigma^+(a^3) - \sigma^-(a^3)}
\end{equation*}
such that:
\begin{equation*}
\left\{
\begin{array}{rcl}
	4 p_{x,1} + 64 p_{x,2} + 776 p_{x,3}
		- p_{y,1} - 15 p_{y,2} - 231 p_{y,3}\\
	\multicolumn{1}{r}{+ \sigma^+(a^1) - \sigma^-(a^1)
		- \sigma^+(a^2) + \sigma^-(a^2)} & > & 0,\\
	- p_{x,1} - 13 p_{x,2} - 19 p_{x,3}
		+ 3 p_{y,1} + 39 p_{y,2} + 387 p_{y,3}\\
	\multicolumn{1}{r}{+ \sigma^+(a^2) - \sigma^-(a^2)
		- \sigma^+(a^3) + \sigma^-(a^3)} & > & 0,\\
	p_{x,0} & = & 0,\\
	p_{y,0} & = & 0,\\
	10 p_{x,1} + 100 p_{x,2} + 1000 p_{x,3}
		+ 10 p_{y,1} + 100 p_{y,2} + 1000 p_{y,3} & = & 1,\\
	p_{x,1} & = & q_{0,0},\\
	2 p_{x,2} & = & q_{0,1} + q_{1,0},\\
	3 p_{x,3} & = & q_{1,1},\\
	p_{y,1} & = & r_{0,0},\\
	2 p_{y,2} & = & r_{0,1} + r_{1,0},\\
	3 p_{y,3} & = & r_{1,1},
\end{array}
\right.
\end{equation*}
with:
\begin{equation*}
\left\{
\begin{array}{rcl}
	Q, R & \multicolumn{2}{l}{PSD,}\\
	\sigma^+(a^1), \sigma^-(a^1),
		\sigma^+(a^2), \sigma^-(a^2),
		\sigma^+(a^3), \sigma^-(a^3),& \geq & 0.\\
\end{array}
\right.
\end{equation*}

\section{Cholesky factorization}
\label{section-cholesky}

The factorization of Cholesky consists in decomposing a positive
semi-definite matrix $M$ into the product of a lower triangular
matrix $L$ and its transpose $L^\transp$.
Formally it reads:
\begin{equation}
	M = LL^{\transp}.
\label{eq-cholesky}
\end{equation}

The decomposition works as follows.
For a matrix $a$ of size $d \times d$, Equation \eqref{eq-cholesky} reads:
\begin{align*}
	M & = \begin{pmatrix}
	m_{1,1} & m_{1,2} & m_{1,3} & \cdots & m_{1,d}\\
	m_{2,1} & m_{2,2} & m_{2,3} & \cdots & m_{2,d}\\
	m_{3,1} & m_{3,2} & m_{3,3} & \cdots & m_{3,d}\\
	\vdots & \vdots & \vdots & \ddots & \vdots\\
	m_{d,1} & m_{d,2} & m_{d,3} & \cdots & m_{d,d}
	\end{pmatrix}\\
	& = \begin{pmatrix}
	l_{1,1} & 0 & 0 & \cdots & 0\\
	l_{2,1} & l_{2,2} & 0 & \cdots & 0\\
	l_{3,1} & l_{3,2} & l_{3,3} & \cdots & 0\\
	\vdots & \vdots & \vdots & \ddots & \vdots\\
	l_{d,1} & l_{d,2} & l_{d,3} & \vdots & l_{d,d}
	\end{pmatrix}
	\cdot
	\begin{pmatrix}
	l_{1,1}	& l_{2,1} & l_{3,1} & \cdots & l_{d,1}\\
	0 & l_{2,2} & l_{3,2} & \cdots & l_{d,2}\\
	0 & 0 & l_{3,3} & \cdots & l_{d,3}\\
	\vdots & \vdots & \vdots & \ddots & \vdots\\
	0 & 0 & 0 & \cdots & l_{d,d}
	\end{pmatrix}\\
	& = \begin{pmatrix}
	l_{1,1}^2 & & & & \text{(symmetric) }\\
	l_{2,1} l_{1,1} & l_{2,1}^2 + l_{2,2}^2\\
	l_{3,1} l_{1,1} & l_{3,1} l_{2,1} + l_{3,2} l_{2,2}
			& l_{3,1}^2 + l_{3,2}^2 + l_{3,3}^2\\
	\vdots & \vdots & \vdots & \ddots\\
	l_{1,1} l_{d,1} & l_{2,1} l_{d,1} + l_{2,2} l_{d,2}
			& l_{3,1} l_{d,1} + l_{3,2} l_{d,2} + l_{3,3} l_{d,3}
			& \cdots & \sum_{i = 1}^d l_{d,i}^2
	\end{pmatrix}.
\end{align*}
The value $m_{i,i}$ and $m_{i,j}$ can be expressed as follows:
\begin{equation*}
	m_{i,i} = \sum_{k = 1}^i l_{i,k}^2
	\text{\qquad and \qquad}
	m_{i,j} = \sum_{k = 1}^j l_{i,k} l_{j,k}
\end{equation*}
The value of the variables $l_{i,i}$ and $l_{i,j}$ are then given by
\begin{equation*}
	l_{i,i} = \sqrt{m_{i,i} - \sum_{k = 1}^{i-1} l_{i,k}^2}
	\text{\qquad and \qquad}
	l_{i,j} = \frac{1}{m_{i,i}}
		\left(m_{i,j} - \sum_{k = 1}^{j - 1} l_{i,k} l_{j,k}\right)
\end{equation*}

\section{Detailed results of the experiments}
\label{section-detailed_results}

Figure \ref{fig-plot_expe_utasp_1piece_test} and
\ref{fig-plot_expe_utasp_5criteria_1piece_test} show the average Spearman
distance and Kendall Tau of the test set after running the experiment
described in Section \ref{section-experiments} with \utap{}.

\begin{figure}[!h]
	\centering
	\includegraphics{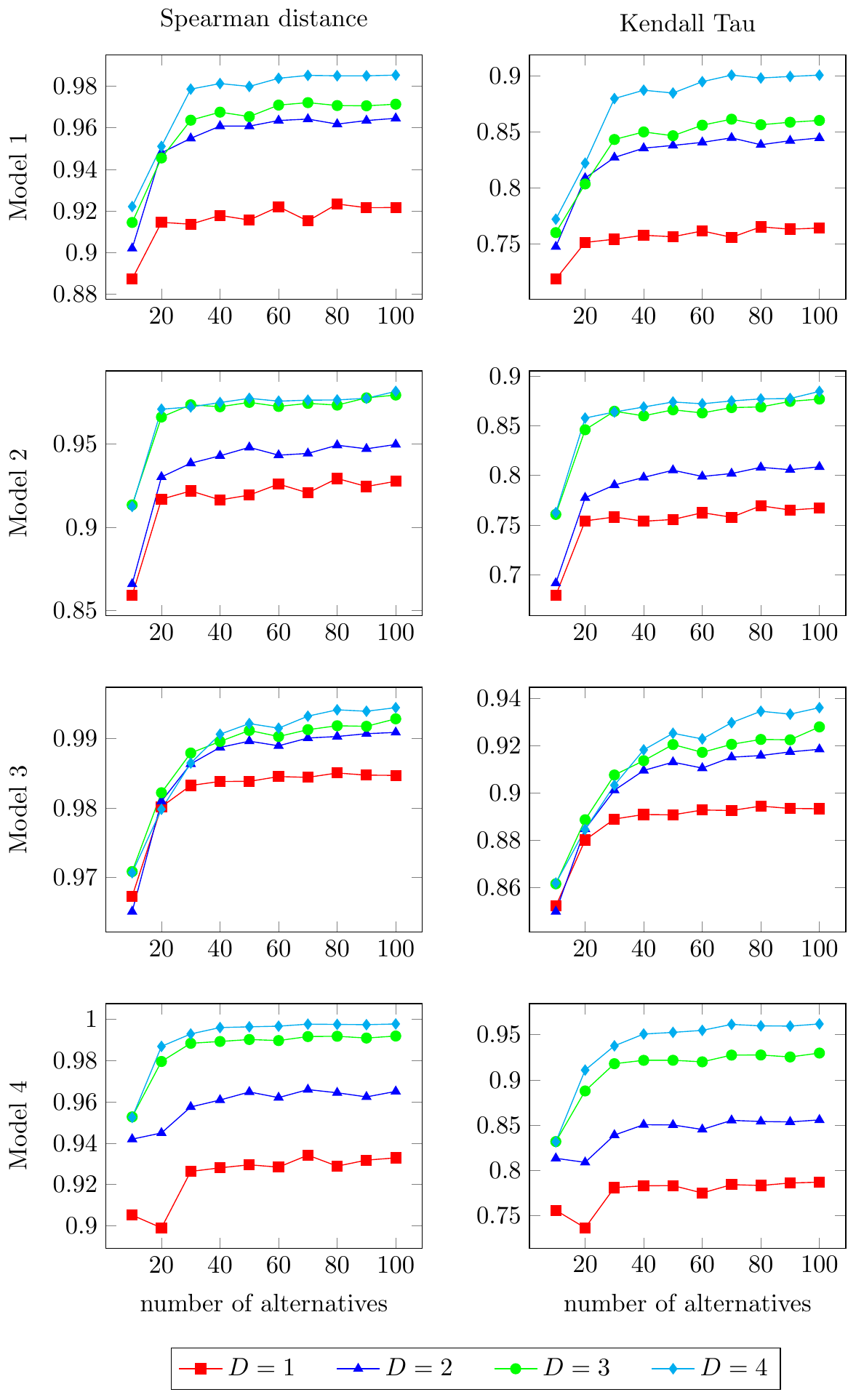}
	\caption{Average Spearman distance and Kendall Tau of the test set
	of models 1 to 4 learned by \utap{} when the
	degree of the marginals vary between 1 and 4.}
	\label{fig-plot_expe_utasp_1piece_test}
\end{figure}

\begin{figure}[!h]
	\centering
	\includegraphics{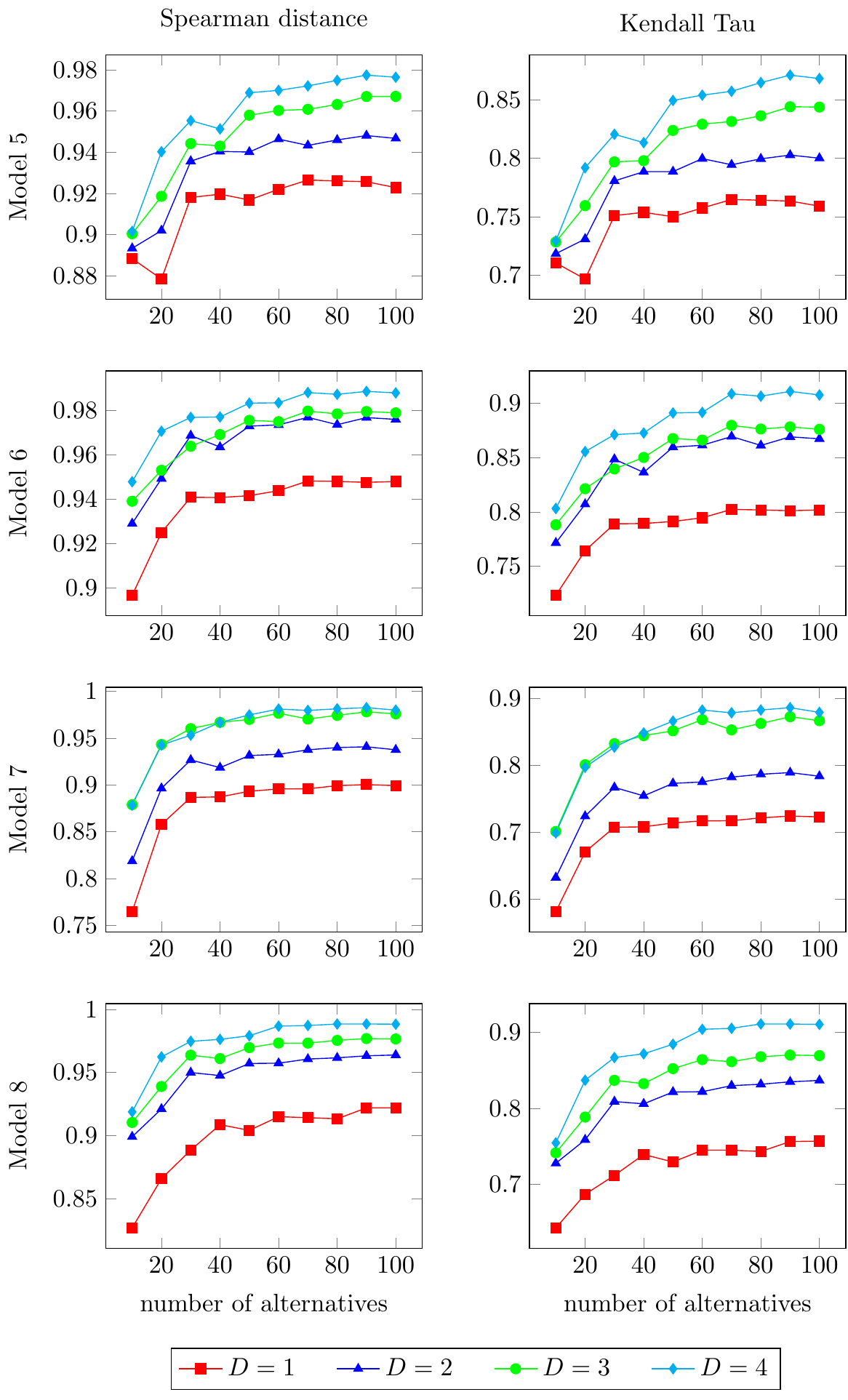}
	\caption{Average Spearman distance and Kendall Tau of the test set
	of models 5 to 8 learned by \utap{} when the
	degree of the marginals vary between 1 and 4.}
	\label{fig-plot_expe_utasp_5criteria_1piece_test}
\end{figure}

Figure \ref{fig-plot_expe_utasp_npieces_test} shows the average Spearman
distance and Kendall Tau obtained with \utasp{} for the four model
composed of 3 criteria presented at Figure
\ref{fig-plot_expe_models_marginals}.
The learned models are composed of polynomials of the third degree which
are continuous up to the second derivative at the connection points.
The number of piece per value function varies between 1 and 5.
Similarly, Figure \ref{fig-plot_expe_utasp_npieces_5criteria_test} shows
the average Spearman distance and Kendall Tau obtained with the four model
composed of 5 criteria.

%Figure \ref{fig-plot_expe_utasp_npieces_test_avg} shows that increasing the
%number of pieces usually helps to increase the quality of the ranking in
%generalization but not always.
%For model 2, we observe that the quality is largely improved when \utasp{}
%learns 3 polynomials per marginal.
%It is also the case for model 4 but, compared to $k=2$, the gain is not so
%important.
%For model 1 and 3, we see that passing from 1 to 2 piece per marginal
%improves the Spearman distance and Kendall Tau.
%However, adding one more piece slightly degrades the quality of the
%ranking.
%For model 3, we also observe that passing from 4 pieces to 5 pieces per
%marginal degrades the solution.
%However we also observe that sometimes the performances decrease when
%the number of pieces per marginal increases.
%This phenomenon is due to the position of the breakpoints which are fixed
%at pre-defined positions.
%One has to be cautious to the position of the breakpoints when increasing
%the number of pieces per marginal in a \utasp{} model.

\begin{figure}[!h]
	\centering
	\includegraphics{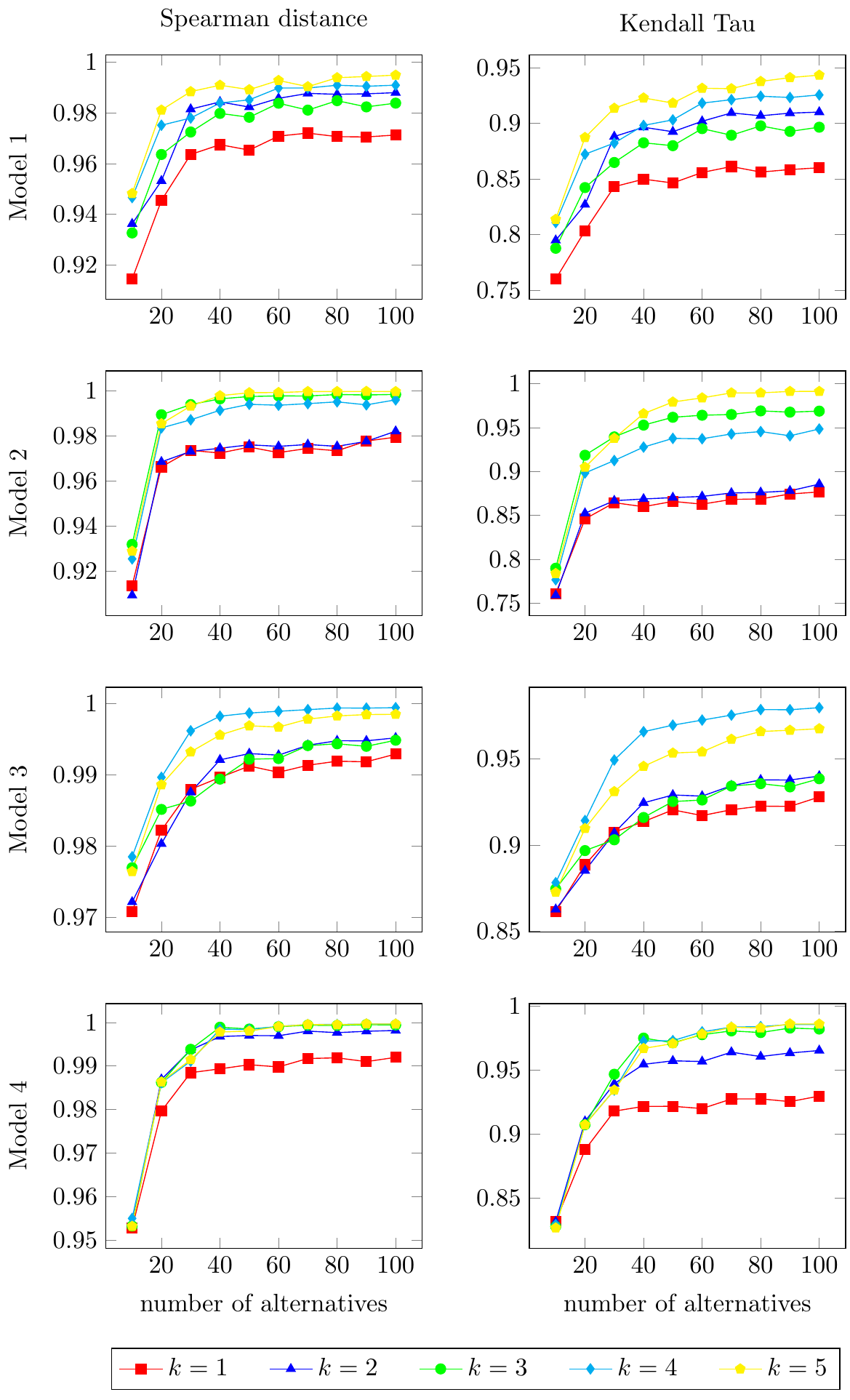}
	\caption{Average Spearman distance and Kendall Tau of the test set
	of models 1 to 4 learned by \utasp{} with
	marginals composed of polynomials of the third degree. The
	continuity at the breakpoints is ensured up to the second
	derivative.}
	\label{fig-plot_expe_utasp_npieces_test}
\end{figure}

\begin{figure}[!h]
	\centering
	\includegraphics{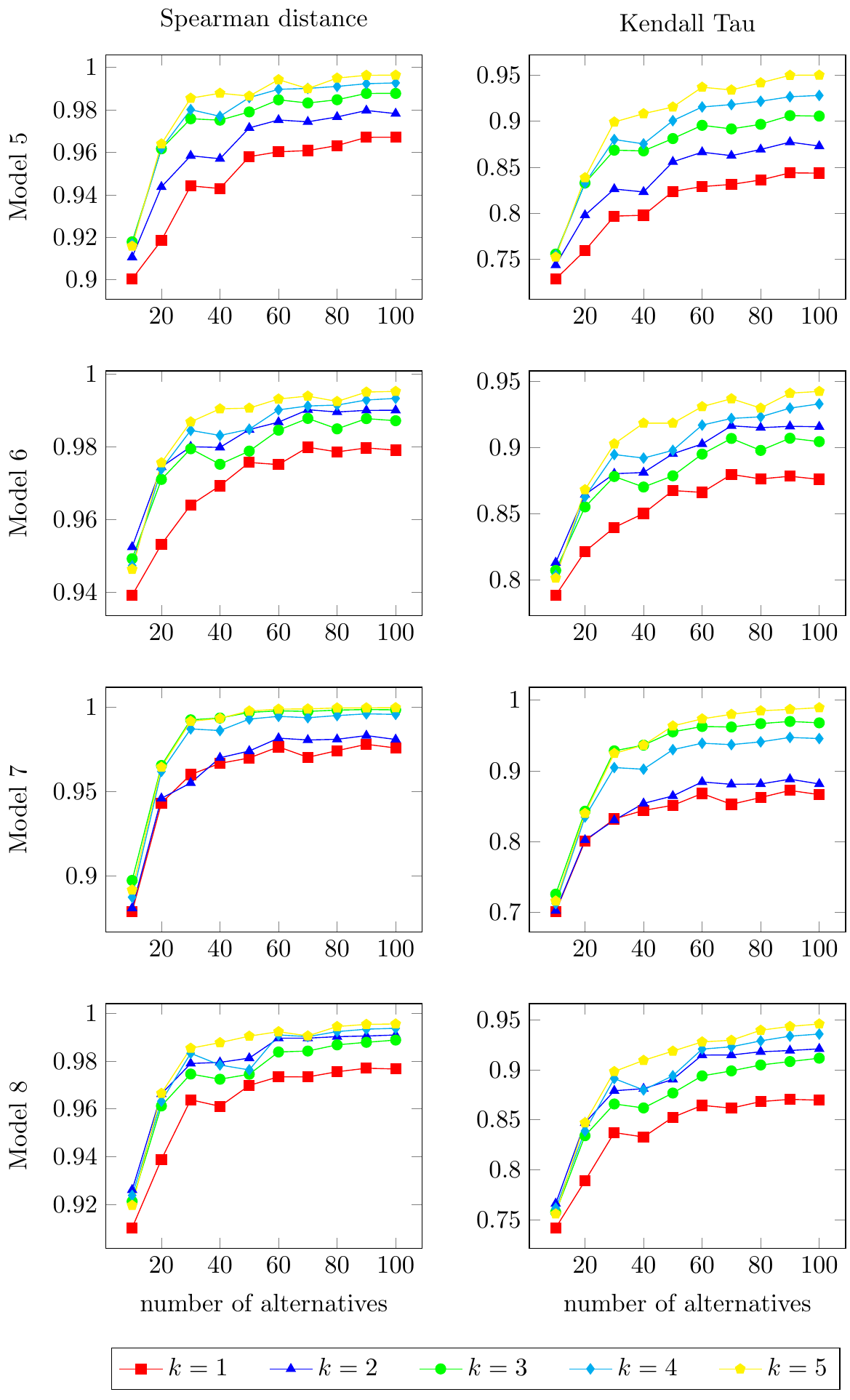}
	\caption{Average Spearman distance and Kendall Tau of the test set
	of models 5 to 8 learned by \utasp{} with
	marginals composed of polynomials of the third degree. The
	continuity at the breakpoints is ensured up to the second
	derivative.}
	\label{fig-plot_expe_utasp_npieces_5criteria_test}
\end{figure}